\documentclass[11pt,reqno,a4paper]{amsart}
\usepackage[a4paper,textwidth=6.15in,textheight=9.05in,centering]{geometry}
\usepackage{amsmath,amssymb,amsthm,mathtools,mathrsfs}
\usepackage{microtype}
\usepackage[T1]{fontenc}
\usepackage{lmodern}
\usepackage{booktabs,array}
\usepackage{enumitem}
\usepackage[hidelinks]{hyperref}
\hypersetup{
 pdftitle={Nondegeneracy of the Sadovskii vortex},
 pdfauthor={Manuel del Pino, Monica Musso, Juncheng Wei}
}
\numberwithin{equation}{section}
\newtheorem{theorem}{Theorem}
\newtheorem{proposition}{Proposition}[section]
\newtheorem{lemma}[proposition]{Lemma}
\newtheorem{corollary}[proposition]{Corollary}
\theoremstyle{definition}

\theoremstyle{remark}

\newcommand{\R}{\mathbb R}

\newcommand{\dd}{\,\mathrm d}

\newcommand{\Ker}{\operatorname{Ker}}
\newcommand{\ind}{\operatorname{ind}}

\newcommand{\cF}{\mathcal F}

\allowdisplaybreaks[1]
\title[Nondegeneracy of the Sadovskii vortex]{Nondegeneracy of the Sadovskii vortex}
\author[M.~del Pino]{Manuel del Pino}
\author[M.~Musso]{Monica Musso}
\author[J.~Wei]{Juncheng Wei}
\date{}
\begin{document}
\begin{abstract}
 We prove nondegeneracy, modulo horizontal translation, of the normalized Sadovskii vortex patch. Building on the free-boundary and root-map framework developed by Huang and Tong, we first obtain a precise quantitative approximation of an exact Sadovskii profile, trapping its free boundary between two explicit barriers constructed around a rational approximation. This localization is then used to control the natural weighted quadratic form at the touching configuration. In the odd class, a ground-state identity shows that horizontal translation is the only zero mode. In the even class, dilation gives a negative direction, while a finite-rank comparison on a compact part of the boundary together with analytic contact-tail estimates yields a positive gap on a codimension-one subspace. Consequently, the Sadovskii dipole is locally rigid at the linearized level once translation is fixed. The computer-assisted part is confined to rigorous interval verification of a finite collection of explicit inequalities; the continuum, contact and operator reductions are proved analytically.
\end{abstract}
\maketitle

\section{Introduction}
A classical problem for the two-dimensional incompressible Euler equation is
to understand steadily translating vortex patches and the linearized
operators around them.  We write the vorticity equation in the form
\begin{equation}\label{eq:euler}
 \partial_t\omega+v\cdot\nabla\omega=0,
 \qquad v=\nabla^\perp\phi,
 \qquad -\Delta\phi=\omega\quad\hbox{in }\R^2,
\end{equation}
where $\nabla^\perp\phi=(\partial_2\phi,-\partial_1\phi)$.  The Sadovskii
dipole is formed by two contiguous regions of opposite constant vorticity.
We write
\begin{equation}\label{eq:patch}
 \omega_*={\bf1}_{\Omega_+}-{\bf1}_{\Omega_-},\qquad
 \Omega_- =\{(x_1,-x_2):(x_1,x_2)\in\Omega_+\},
\end{equation}
and assume that the upper component is symmetric with respect to the vertical
axis.  We fix the vorticity strengths to be $+1$ and $-1$, center the configuration,
and normalize the two endpoints of the free arc to $x_1=\pm1$.  The free arc
is then the graph of a positive even function $f$ on $(-1,1)$ with
$f(\pm1)=0$.  If the dipole translates with speed $W>0$, the relative stream
function is $\psi=\phi-Wx_2$, and the free-boundary equation is simply
\begin{equation}\label{eq:fb-condition}
 \psi=0\qquad\hbox{on the free arc }\Gamma.
\end{equation}

Sadovskii introduced this touching, steadily translating configuration in
\cite{Sadovskii}.  Counter-rotating vortex pairs are among the classical
coherent structures of two-dimensional ideal flow; for broad accounts of
their dynamics and instabilities we refer to Saffman~\cite{Saffman} and to the
review of Leweke, Le~Diz\`es and Williamson~\cite{LewekeLeDizesWilliamson}.
The mathematical theory of steady planar pairs has a long history.  Among the
classical existence results are the stream-function construction of
Norbury~\cite{Norbury}, the variational theory of Turkington~\cite{Turkington}
and Burton~\cite{Burton1988}, and the later stability theory for variational
vortex pairs~\cite{BurtonNussenzveigLopesLopesFilho,Burton2021}.  Perturbative
and desingularization approaches have produced local and global families of
translating and rotating vortex patches; see, for example,
\cite{HmidiMateu,HassainiaHmidi,CaoLaiZhan,HassainiaWheeler,GarciaHaziot}.
The Lamb dipole, another distinguished translating dipole, has a complementary
variational stability theory~\cite{AbeChoi}.

The Sadovskii vortex is the limiting touching member of this family.  Its
rigorous existence was established only recently, and by two rather different
methods.  Choi, Jeong and Sim~\cite{ChoiJeongSim} obtained a variational
construction based on energy maximization under an impulse constraint, while
Huang and Tong~\cite{HuangTong} developed a free-boundary fixed-point
construction.  A scaling-invariant variational framework connecting the patch
to smoother Sadovskii vortices, together with stability information, was
subsequently developed in~\cite{AbeChoiJeongSimWoo}.  Related uniqueness and
stability questions for concentrated traveling pairs have also been studied in
\cite{CaoQinZhanZou}.

Our existence step is deliberately localized and should be distinguished from
these global existence theories.  We use the root map introduced by
Huang--Tong, but prove here that a very explicit narrow order interval is
invariant under that map.  The barrier inequalities are quantitative and hold
on the whole free boundary, including the contact limit.  Schauder's theorem
then produces an exact normalized Sadovskii profile inside this explicit
neighbourhood. In this paper we give a precise quantitative description of the Sadovskii vortex, localizing an exact profile within an explicit band of size \(O(10^{-3})\) around a rational approximation. This localization is obtained using the basic mapping properties of the Huang–Tong root map, so our argument is not intended as an independent existence proof. We also do not claim that the fixed point obtained here coincides with the variational solution of Choi–Jeong–Sim or with any particular fixed point arising in the Huang–Tong construction, since no uniqueness theorem is presently available to identify these solutions.

The main question of this paper begins after existence.  A coherent vortex
may exist and still possess hidden infinitesimal deformations through nearby
equilibria.  Because horizontal translations preserve a traveling dipole, the
linearized free-boundary operator necessarily has a translation mode in its
kernel.  The natural nondegeneracy statement is that this is the entire
kernel.  Equivalently, after fixing the translation gauge, the linearized
problem is invertible in the natural weighted space.  This is a local rigidity
property: it does not imply global uniqueness, and by itself it is not a
nonlinear stability theorem.

Nondegeneracy is important independently of any particular application.  It
is the structural hypothesis that permits one to solve the linearized equation
and hence to use a coherent vortex as a building block in perturbative,
continuation and desingularization arguments.  Closely related mechanisms are
visible in the desingularization of nondegenerate point-vortex equilibria into
vortex patches~\cite{HassainiaWheeler}, in global constructions of hollow
vortices from nondegenerate configurations~\cite{ChenWalshWheeler}, and, more
directly at the level of a vortex patch itself, in the recent desingularization
theory of Radu and Stevenson~\cite{RaduStevenson}, where nondegeneracy of a
steady rotating patch is the hypothesis that makes the nearby construction
possible.  Perturbative, desingularization, and gluing constructions for Euler
vortices provide further background for this point of view; see, for example,
\cite{SmetsVanSchaftingen,ddmw,ddmps}.  Our geometry is substantially different
from those perturbative settings: the two components touch, the profile is not
explicit, and the coefficient of the local linearized term degenerates at the
contact points.

The theorem below gives more information than the absence of additional zero
modes.  In the odd class the quadratic form is nonnegative and its nullspace
is exactly the translation direction.  In the even class there is exactly one
negative direction, supplied by dilation, and no kernel; on a suitable
codimension-one subspace we obtain a quantitative positive gap.  Thus the
result describes the basic Morse structure of the Sadovskii dipole at the
linearized level.

There are two features of this problem that make the usual nondegeneracy
argument less direct.  First, the boundary of the vortex is not explicit.
One cannot therefore write the linearized operator around a known closed
form and estimate it afterwards.  Second, the coefficient of the local part
of the linearized boundary equation vanishes at the contact endpoints.  A
naive division by this coefficient produces an artificial singularity and
obscures the natural energy space.

The first issue is handled by localizing an exact solution rather than by
approximating the linearized operator at a merely numerical profile.  We use
the root map of Huang--Tong and construct two explicit graphs which are
respectively a lower and an upper barrier.  The barrier inequalities are
proved on the whole parameter interval, including the contact limit.  The
root map sends the order interval between the barriers into itself, and
Schauder's theorem then yields an exact free boundary in that interval.  The
role of the rational approximation is only to make this order interval
explicit and narrow enough for the later operator estimates.

Once an exact profile has been obtained, the existence construction is no
longer used.  The second half of the paper starts again from an arbitrary
exact profile in the same barrier class.  The coefficient which vanishes at
contact is incorporated into the Hilbert-space norm.  In this weighted space
the local part of the second variation becomes the identity, while the Green
interaction is a compact self-adjoint perturbation.  Translation produces an
odd zero mode.  A ground-state identity shows that there are no others in the
odd class.  In the even class, dilation produces a strictly negative
direction; a comparison with one fixed reference operator shows that there
can be no second nonpositive direction.  This gives the desired kernel
statement.

The quantitative part of the proof is deliberately kept separate from these
analytic arguments.  What must be checked numerically is finite: two barrier
signs, one finite-dimensional spectral inequality, two operator comparison
errors, a positive denominator, and a small number of contact-tail bounds.
Every numerical input needed to follow the proof is stated in the main text
before it is used.  The remaining exact data and the complete reproducibility
record belong to the accompanying computational supplement.

\subsection{Outline of the proof}

We describe the proof here from beginning to end before entering the detailed
estimates.  There are two separate parts.  First we prove that an exact
Sadovskii profile lies between two explicit barriers.  Then we prove
nondegeneracy for every exact profile between those barriers.

\paragraph{1. The profile and the barriers.}
We choose an explicit reference graph $\bar f$ and two explicit graphs
$f_-$ and $f_+$ with
\[
        f_-\le \bar f\le f_+.
\]
The class $\mathcal K$ consists of the even decreasing graphs lying between
$f_-$ and $f_+$.  The formulas for these three graphs are written immediately
below this outline.  In the main proof we work directly with the physical
graphs $\bar f$ and $f_\pm$; no additional unknown profile is introduced.

For each $f\in\mathcal K$, the region below its graph produces a stream
function $\phi_f$ and a travelling speed $c(f)$.  Keeping this region fixed,
we solve
\[
        \phi_f(x,y)-c(f)y=0
\]
for $y>0$ at each horizontal position $x$.  The resulting graph is denoted by
$\mathcal R(f)$.  Thus $f$ is an exact travelling vortex exactly when
\[
        \mathcal R(f)=f.
\]
To prove that $\mathcal R$ maps $\mathcal K$ into itself, it is enough to show
that the left-hand side above has opposite signs at $y=f_-(x)$ and
$y=f_+(x)$.  We rewrite those two sign conditions as the positivity of two
scalar functions $m_-(T)$ and $m_+(T)$.  Once their positivity is proved for
$0<T\le1$, Schauder's theorem gives an exact profile in $\mathcal K$.

The estimates for $m_\pm$ are divided according to the geometry.  Very close
to contact we first cancel the two large logarithmic terms analytically.  In
the middle part we obtain lower bounds at finitely many target values and use
a one-sided second-derivative estimate to fill the intervals between them.
Near the symmetry axis we use the physical horizontal coordinate, where the
formulas are regular.

\paragraph{2. The linearized equation.}
From this point on the construction of the exact profile is no longer used.
We fix any exact $f\in\mathcal K$.  If $h$ is a normal displacement of the
free boundary, differentiation of the boundary equation gives
\[
        L_fh=q_fh+S_fh,
        \qquad
        S_fh(X)=\int_{\Gamma_f}G_H(X,Y)h(Y)\,\dd s_Y,
\]
where $q_f=\partial_\nu\psi_f<0$ on the open free arc.  The coefficient $q_f$
tends to zero at the contact points.  We therefore use the norm
\[
        \|h\|_{\mathcal H_f}^2
        =\int_{\Gamma_f}(-q_f)h^2\,\dd s
\]
instead of dividing by $q_f$.

Put
\[
        a=(-q_f)^{1/2},\qquad u=ah.
\]
Then
\[
        \mathcal Q_f[h]=\langle(I-K)u,u\rangle,
\]
where
\[
        (Ku)(X)=\int_{\Gamma_f}
        \frac{G_H(X,Y)}{a(X)a(Y)}u(Y)\,\dd s_Y.
\]
The operator $K$ is compact and self-adjoint.  Hence $L_fh=0$ is equivalent,
in the completed form space, to
\[
        Ku=u.
\]
The problem is therefore to determine all eigenfunctions of $K$ with
eigenvalue $1$.

\paragraph{3. Odd and even perturbations.}
Horizontal translation gives the normal displacement $e_1\cdot\nu$.  It is
odd and satisfies $L_f(e_1\cdot\nu)=0$.  On the right half of the free
boundary this function is positive.  Using this positive solution in the
quadratic form gives an identity whose right-hand side is an integral of a
square with a positive kernel.  It follows that the quadratic form is
nonnegative on odd perturbations and vanishes only for horizontal
translation.

Dilation is even.  Differentiating the family of dilated vortices gives an
even displacement with strictly negative quadratic form.  Thus the even
part has at least one negative direction.  We must prove that after one
linear condition is imposed, every remaining even direction has positive
quadratic form.

\paragraph{4. The part of the boundary away from contact.}
Choose a fixed point $x_c<1$ and first restrict the even problem to
$0\le x\le x_c$.  On this compact part of the boundary the coefficient
$w_f=-\psi_{f,y}(x,f(x))$ stays positive, so there is no degeneracy from the
contact points.

We put all exact profiles in one fixed coordinate and compare three
operators:
\[
        A_f^{\rm core},\qquad
        A_{\bar f}^{\rm core},\qquad
        \mathsf M.
\]
The first is the integral operator for the exact profile $f$ on the compact
part of the boundary.  The second is the same operator for the explicit
reference graph $\bar f$.  The third acts on piecewise constant functions on
a fixed finite partition and is represented by a real symmetric matrix.

The matrix calculation shows that, after removing one direction, the
quadratic form of $\mathsf M$ is strictly below the critical level $1$.  We
then estimate directly the difference between $\mathsf M$ and
$A_{\bar f}^{\rm core}$, and the difference between
$A_{\bar f}^{\rm core}$ and $A_f^{\rm core}$.  Finally we prove a positive
lower bound for the denominator in the corresponding quotient.  These three
estimates transfer the matrix inequality to the exact profile.  Therefore,
after one linear condition is imposed on the compact part of the boundary,
\[
        \langle A_f^{\rm core}v,v\rangle
        \le \alpha\|v\|^2
\]
with a number $\alpha<1$ independent of the exact profile in the barrier
class.

The finite matrix is used only for this step.  We never approximate an
eigenfunction of the exact Sadovskii vortex numerically.

\paragraph{5. Restoring the contact regions.}
We next put back the parts of the boundary between $x_c$ and the two contact
points.  Near contact we use
\[
        t=-\log f(x).
\]
In this variable the Green kernel has a simple explicit upper bound.  From
it we obtain a bound for the interaction of two functions supported near
contact and a separate bound for the interaction between the compact part
and the contact part.

Writing an even function as
\[
        u=u_{\rm c}+u_{\rm t},
\]
where $u_{\rm c}$ is supported on the compact part and $u_{\rm t}$ near
contact, the three estimates give
\[
 \langle(I-K)u,u\rangle
 \ge (1-\alpha)\|u_{\rm c}\|^2
      -2\beta\|u_{\rm c}\|\,\|u_{\rm t}\|
      +(1-\gamma)\|u_{\rm t}\|^2,
\]
with $\alpha<1$ and $\gamma<1$.  Thus the full problem is reduced to checking
that the two-variable quadratic form on the right is positive.  The
quantitative estimates below give precisely this inequality.

Consequently the even quadratic form is positive after one linear condition
is imposed.  Since dilation already gives one strictly negative even
direction, the even kernel is trivial.  Combining this with the odd argument
gives
\[
        \Ker_{\rm form}L_f=\operatorname{span}\{e_1\cdot\nu\}.
\]

The numerical work is used only to certify the explicit inequalities needed
at the points indicated above.  The logical steps from those inequalities to
existence and nondegeneracy are contained in the main text.

\subsection{The profile enclosure}
For $|x|<1$, set
\begin{equation}\label{prof:eq:T-main}
T(x)=\frac{1}{1-\log(1-x^2)},\qquad
 \omega_0(x)=(1-x^2)\bigl(1-\log(1-x^2)\bigr).
\end{equation}
The coordinate $T$ equals one at $x=0$ and tends to zero at contact.  Let
$P_5,Q_5$ be the fixed degree-five polynomials and let $\rho$ be the fixed
piecewise-affine function listed with the exact data in
Section~\ref{app:profile-details}.  We define directly the physical reference
graph and the two barriers by
\begin{equation}\label{eq:intro-reference}
 \bar f(x)=\omega_0(x)\frac{P_5(T(x))}{Q_5(T(x))},\qquad
 f_\pm(x)=\bar f(x)\pm\omega_0(x)\rho(T(x)).
\end{equation}
Thus $\bar f$ is the graph used throughout the main argument; no separate
unknown function $g$ is needed.  We consider the closed convex class
\begin{equation}\label{prof:eq:Kset-main}
\mathcal K=\{f\in C([-1,1]):f\hbox{ is even and nonincreasing on }[0,1],
                              \ f_-\le f\le f_+\}.
\end{equation}

\medskip 
Our first main result gives a quantitative localization of an exact
Sadovskii profile.

\begin{proposition}[Quantitative localization]\label{prof:thm:main}
There exists an exact normalized Sadovskii profile $f_*\in\mathcal K$.
It is positive in $(-1,1)$, strictly decreasing on $(0,1)$, and satisfies
\begin{equation}\label{prof:eq:main-band-new}
 |f_*(x)-\bar f(x)|
 \le \omega_0(x)\rho(T(x)),\qquad
 \|f_*-\bar f\|_{L^\infty(-1,1)}
 \le 6\times10^{-3}.
\end{equation}
\end{proposition}

The pointwise bound in \eqref{prof:eq:main-band-new} retains the vanishing
contact scale. Near $|x|=1$, both barriers have height comparable to
$(1-|x|)\log(1/(1-|x|))$. Section~\ref{prof:sec:profile-enclosure} reduces the
proposition to two scalar inequalities for the exact root map of Huang--Tong.
Section~\ref{sec:profile-estimates} proves those inequalities.

\subsection{The linearized form and the main theorem}
Let $G_H$ be the half-plane Green function defined by \begin{equation}\label{prof:eq:GH}
 G_H((x,y),(\xi,\eta))=\frac1{4\pi}
 \log\frac{(x-\xi)^2+(y+\eta)^2}{(x-\xi)^2+(y-\eta)^2}.
\end{equation}
Fix an exact profile $f\in\mathcal K$, and write
\[
 \Omega_f=\{(x,y):-1<x<1,\ 0<y<f(x)\},\qquad
 \Gamma_f=\{(x,f(x)):-1<x<1\}.
\]
We denote by
\begin{equation}\label{prof:eq:phi}
 \phi_f(x,y)=\int_{-1}^1\int_0^{f(\xi)}G_H((x,y),(\xi,\eta))\,\dd\eta\dd\xi.
\end{equation}
the stream function generated by the upper patch. 
Let $c(f)>0$ denote the travelling speed associated with $f$,
\begin{equation}\label{prof:eq:speed}
 c(f)=\frac1{2\pi}\int_{-1}^1
              \log\left(1+\frac{f(\xi)^2}{(1-\xi)^2}\right)\,\dd\xi.
\end{equation}
We then set
\[
 \psi_f(x,y)=\phi_f(x,y)-c(f) \, y.
\]
Thus $\psi_f$ is the relative stream function in the moving frame, and the
free-boundary equation is
\[
 \psi_f=0\qquad\hbox{on }\Gamma_f.
\]
Since $\Omega_f$ is bounded and the gradient of the logarithmic kernel is
locally integrable, $\nabla\phi_f$ is continuous; in particular, $\psi_f$ is
$C^1$ up to the closed upper half-plane. A proof of this elementary
regularity, including the contact points, is given in
Section~\ref{prof:sec:profile-enclosure}.

As shown there, an exact profile is $C^1$ on $(-1,1)$ and
$\nabla\psi_f\ne0$ on the open free arc. Let $\nu$ be the outward unit normal
to $\Omega_f$ along $\Gamma_f$. We write
\begin{equation}\label{eq:qnegative}
 q_f=\partial_\nu\psi_f<0\qquad\hbox{on }\Gamma_f.
\end{equation}
For later use, let $e_1=(1,0)$.

For a normal displacement $h:\Gamma_f \to \R$, we define
\begin{equation}\label{eq:intro-form}
 \mathcal Q_f[h]=\int_{\Gamma_f}(-q_f)h^2\,\dd s
 -\iint_{\Gamma_f\times\Gamma_f}G_H(X,Y)h(X)h(Y)\,\dd s_X\dd s_Y.
\end{equation}
The natural space is
\begin{equation}\label{eq:intro-weighted-space}
 \mathcal H_f=\{h:(-q_f)^{1/2}h\in L^2(\Gamma_f)\},\qquad
 \|h\|_{\mathcal H_f}^2=\int_{\Gamma_f}(-q_f)h^2\,\dd s.
\end{equation}
The form extends continuously to this space. No zero trace is imposed at the
contact points. We denote by $\mathcal H_{f,e}$ and $\mathcal H_{f,o}$ the
closed even and odd subspaces of $\mathcal H_f$ under reflection in the
vertical axis.

\medskip
Our second main estimate is the following.

\begin{proposition}[Even coercivity]\label{prop:principal-coercivity}
For every exact Sadovskii profile
$f\in\mathcal K$, there is a closed
codimension-one subspace $X_f$ of the even space $\mathcal H_{f,e}$ such that
\begin{equation}\label{eq:principal-coercivity}
 \mathcal Q_f[h]\ge\frac1{250}\|h\|_{\mathcal H_f}^2,
 \qquad h\in X_f.
\end{equation}
\end{proposition}
The codimension-one space $X_f$ is defined by imposing a single condition
on the compact core, coming from the finite-rank comparison. Once this
condition is imposed, the tail component is arbitrary. In particular,
$X_f$ is not defined as the orthogonal complement of the dilation mode,
and in general the two subspaces need not coincide.
We construct it in Sections~\ref{sec:operator-comparison} and
\ref{sec:core-tail-abstract}.

\medskip
Let
\begin{equation}\label{eq:intro-L-simple}
 L_fh=q_fh+S_fh,
 \qquad S_fh(X)=\int_{\Gamma_f}G_H(X,Y)h(Y)\,\dd s_Y, \quad X \in \Gamma_f.
\end{equation}
This is the shape derivative of \eqref{eq:fb-condition} at fixed travelling
speed. Its form realization is defined in Section~\ref{sec:selfadjoint}. We
write $\Ker_{\rm form}L_f$ for the kernel of the bilinear form associated
with $\mathcal Q_f$; its precise completed-space definition is given there.

\medskip
The two propositions above lead to the main theorem.
\begin{theorem}\label{thm:main}
An exact normalized Sadovskii profile exists in $\mathcal K$.
For every such profile, the moving-contact form kernel of $L_f$ is
\begin{equation}\label{eq:main-kernel}
 \Ker_{\rm form}L_f=\operatorname{span}\{e_1\cdot\nu\}.
\end{equation}
The quadratic form is nonnegative on the odd space, with equality only
in the translation direction. On the even space it has exactly one
negative direction and no kernel.
\end{theorem}

The assertion is for the band \eqref{prof:eq:Kset-main} and the weighted
space \eqref{eq:intro-weighted-space}. It does not assert global uniqueness
of Sadovskii dipoles or dynamical stability. On the even fixed-impulse tangent
space the form is strictly positive; this follows from the dilation identity
and is proved in Section~\ref{sec:proof-main-v40}.

\subsection{How the estimates fit together}
The first quantitative statement is Proposition~\ref{prof:thm:main}.  Its
proof is not a convergence argument for a numerical approximation.  The
computer-assisted estimates establish two strict barrier signs for the exact
root map, and the fixed-point theorem then produces an exact Euler patch.
The contact calculation, the transition region and the regular region are
treated differently because the most efficient analytic coordinates are
different in those three places.  The resulting pieces meet on closed
intervals and cover the full parameter range.

For the linearized equation write $a=(-q_f)^{1/2}$ and $u=ah$.  The quadratic
form becomes
\[
 \mathcal Q_f[h]=\langle(I-K)u,u\rangle,
\]
where
\begin{equation}\label{eq:intro-K-simple}
 Ku(X)=\int_{\Gamma_f}\frac{G_H(X,Y)}{a(X)a(Y)}u(Y)\,\dd s_Y
\end{equation}
is compact and self-adjoint.  Horizontal translation satisfies
$L_f(e_1\cdot\nu)=0$ and yields the positive ground state on the odd half-arc.
Dilation satisfies
\[
 L_f(X\cdot\nu)=c(f) \, x_2,
 \qquad
 \mathcal Q_f[X\cdot\nu]=-3c(f) \, \int_{\Omega_f}x_2\,\dd X<0,
\]
and therefore provides the negative even direction.

It remains to exclude a second nonpositive even direction.  This is the only
place where a finite-dimensional comparison is used.  On a compact core the
exact operator is compared with a fixed finite-rank operator; the comparison
is then completed by analytic estimates on the contact tails and on the
core--tail coupling.  Proposition~\ref{prop:principal-coercivity} is the
resulting codimension-one gap.  The numerical values needed for this last step are stated in the main text at the
point where they first enter the argument. Their rigorous interval verification is
summarized briefly at the end of the paper and documented in full in the
computational supplement.

\section{The profile equation and its barriers}
\label{prof:sec:profile-enclosure}
\subsection{The root map}\label{sec:graph-level-map-v39}

We first describe precisely the fixed-point map used in the existence argument.
Let $f\in\mathcal K$ be any admissible graph. At this stage $f$ is only an
input graph; it is not assumed to satisfy the free-boundary equation. It
defines the upper patch
\[
 \Omega_f=\{(x,y):-1<x<1,\ 0<y<f(x)\},
\]
and hence, through \eqref{prof:eq:phi}, the stream function $\phi_f$. Its
travelling speed is the number $c(f)$ in \eqref{prof:eq:speed}, and we set
\[
        \psi_f(x,y)=\phi_f(x,y)-c(f)y.
\]
For each fixed $x\in(-1,1)$ we now keep the patch $\Omega_f$ fixed and look
along the vertical line above $x$ for a positive zero of $\psi_f$. Since
$y>0$, this is equivalent to solving
\begin{equation}\label{prof:eq:Phi}
 \Phi(f;x,y):=\frac{\phi_f(x,y)}{y}-c(f)=0.
\end{equation}
For the admissible graphs considered here this equation has a unique positive
root. We denote it by
\[
        y=\mathcal R(f)(x).
\]
Thus $f$ and $\mathcal R(f)$ have different roles: $f$ determines the vortex
patch and the velocity field, while $\mathcal R(f)$ is the graph obtained by
taking the zero level set of the corresponding relative stream function.

An exact travelling vortex is precisely a fixed point of this map. Indeed,
\[
        \mathcal R(f)=f
\]
is equivalent to
\[
        \psi_f(x,f(x))=0,\qquad -1<x<1,
\]
which is exactly the free-boundary condition \eqref{eq:fb-condition} on
$\Gamma_f$. Hence our existence problem is reduced to finding a fixed point
of $\mathcal R$.

The only general properties of the Huang--Tong root map that we use are the
following. For every admissible input graph, $\Phi$ is strictly decreasing in
$y$; for $x>0$ it is also strictly decreasing in $x$. Consequently the root is
unique, and $\mathcal R(f)$ is even and nonincreasing. Moreover, on any class
with a common lower contact bound and a common upper height bound, the maps
$\mathcal R(f)$ have a uniform $C^{1/2}$ bound and $\mathcal R$ is continuous
in the uniform norm. These properties are proved in
Huang--Tong~\cite[Lemma 3.2, Proposition 3.3 and Propositions 4.9, 4.11]{HuangTong}.

For our explicit barriers, the exact polynomial estimates give
\begin{equation}\label{eq:c35-class-bounds}
 0.04(1-|x|)\le f_-(x)\le f(x)\le f_+(x)\le2,
 \qquad f\in\mathcal K.
\end{equation}
The barriers are continuous and strictly decreasing on $(0,1)$. Therefore
$\mathcal K$ is a nonempty closed convex subset of the even subspace of
$C([-1,1])$. The bounds above give the hypotheses just stated for the root
map, so $\mathcal R$ is continuous on $\mathcal K$ and
$\mathcal R(\mathcal K)$ is relatively compact.

\begin{proposition}[Barrier criterion]\label{prof:prop:abstract-barrier}
Suppose that, for every $f\in\mathcal K$ and $|x|<1$,
\begin{equation}\label{prof:eq:barrier-lower-sign}
 \Phi(f;x,f_-(x))>0,\qquad \Phi(f;x,f_+(x))<0.
\end{equation}
Then $\mathcal R(\mathcal K)\subset\mathcal K$, and $\mathcal R$ has a fixed
point in $\mathcal K$.
\end{proposition}
\begin{proof}
Because $y\mapsto\Phi(f;x,y)$ is strictly decreasing, the two inequalities in
\eqref{prof:eq:barrier-lower-sign} place its unique positive root between
$f_-(x)$ and $f_+(x)$. Thus $\mathcal R(\mathcal K)\subset\mathcal K$.
Continuity and relative compactness of $\mathcal R$ then give a fixed point by
Schauder's theorem.
\end{proof}

\subsection{Reduction of the barrier signs to scalar margins}

We now reduce the two inequalities in
\eqref{prof:eq:barrier-lower-sign} to two scalar functions of the target
parameter $T$. This is the only purpose of the notation introduced in this
subsection. The next section will prove that these scalar functions are
positive on the whole interval $0<T\le1$.

Recall from \eqref{prof:eq:phi} that
\[
 \phi_f(x,y)=\int_{-1}^1\int_0^{f(\xi)}
 G_H((x,y),(\xi,\eta))\,\dd\eta\dd\xi.
\]
Thus $(\xi,\eta)$ are simply integration variables in the vortex patch. Since
$f$ is even, the vertical integration columns above $\xi$ and $-\xi$ have the
same height. We keep these two contributions together throughout the
estimate.

Fix a target point $(x,y)$ and a vertical integration column of height $b$
at horizontal distance $a=x-\xi$. Integrating the Green kernel first in the
vertical variable gives
\[
 \frac1y\int_0^bG_H((x,y),(\xi,\eta))\,\dd\eta
 =\frac1{4\pi y}\int_0^b
 \{\log(a^2+(y+\eta)^2)-\log(a^2+(y-\eta)^2)\}\,\dd\eta.
\]
Introduce the elementary primitive
\begin{equation}\label{prof:eq:J}
 J_a(s)=s\log(a^2+s^2)-2s+2|a|\arctan\frac{s}{|a|},
\end{equation}
with the last term defined continuously at $a=0$. Observe that $J_a$ is odd.
The preceding vertical integral equals
\begin{equation}\label{prof:eq:V}
 V(a,y,b)=\frac{J_a(y+b)+J_a(y-b)-2J_a(y)}{4\pi y}.
\end{equation}
Thus $V(a,y,b)$ is simply the contribution of one vertical integration column
to $\phi_f(x,y)/y$.

The speed $c(f)$ in \eqref{prof:eq:speed} is an integral over the same
horizontal variable. The contribution of one column of height $b$ at distance
$p$ from the right contact point is
\begin{equation}\label{prof:eq:C}
 C(p,b)=\frac1{2\pi}\log\left(1+\frac{b^2}{p^2}\right).
\end{equation}

Now write the target as $x=1-d$ and the positive horizontal integration
coordinate as $\xi=1-t$, $0\le t\le1$. Its symmetric partner is $-\xi$.
The two horizontal separations from the target are $t-d$ and $2-d-t$, and
their common height is $b=f(1-t)$. Hence their combined contribution to
$\Phi$ is exactly
\begin{equation}\label{prof:eq:Fpair}
 \cF(d,y;t,b)
 =V(t-d,y,b)+V(2-d-t,y,b)-C(t,b)-C(2-t,b).
\end{equation}
The same height $b$ occurs in all four terms because the graph is even. This
is why the minimization or maximization below is performed only after the two
symmetric contributions have been combined.

For later estimates it is also useful to record the derivative with respect
to this common height:
\begin{align}
 \partial_b\cF(d,y;t,b)
 ={}&\frac1{4\pi y}\log
 \frac{(t-d)^2+(y+b)^2}{(t-d)^2+(y-b)^2}\notag\\
 &+\frac1{4\pi y}\log
 \frac{(2-d-t)^2+(y+b)^2}{(2-d-t)^2+(y-b)^2}\notag\\
 &-\frac b\pi\left(\frac1{t^2+b^2}
 +\frac1{(2-t)^2+b^2}\right).
 \label{prof:eq:Fb}
\end{align}
A strict sign of \eqref{prof:eq:Fb} allows one to select an endpoint of the
height interval. If no sign is proved, the whole interval of heights is
retained.

For $0<T\le1$, inversion of \eqref{prof:eq:T-main} gives
\[
 E(T)=e^{1-1/T},\qquad
 d(T)=\frac{E(T)}{1+\sqrt{1-E(T)}},\qquad
 \omega_0(T)=\frac{E(T)}T.
\]
Thus $x=1-d(T)$. Define the two target heights and the two possible integration
heights by
\begin{equation}\label{prof:eq:ybarriers}
 y_\pm(T)=f_\pm(1-d(T)),
\end{equation}
\begin{equation}\label{prof:eq:bbarriers}
 b_\pm(t)=f_\pm(1-t).
\end{equation}
For a fixed admissible graph, its actual height $b=f(1-t)$ lies in
$[b_-(t),b_+(t)]$. We therefore define the worst possible paired contributions
\begin{align}
 E_-(T,t)&=\min_{b\in[b_-(t),b_+(t)]}
 \cF(d(T),y_-(T);t,b),\label{prof:eq:Eminus}\\
 E_+(T,t)&=\max_{b\in[b_-(t),b_+(t)]}
 \cF(d(T),y_+(T);t,b),\label{prof:eq:Eplus}
\end{align}
and finally the normalized margins
\begin{align}
 m_-(T)&=\frac1{\omega_0(T)}\int_0^1E_-(T,t)\dd t,
 \label{prof:eq:mminus}\\
 m_+(T)&=-\frac1{\omega_0(T)}\int_0^1E_+(T,t)\dd t.
 \label{prof:eq:mplus}
\end{align}

\begin{lemma}[Margins imply the barrier signs]\label{prof:lem:margins}
If $m_-(T)>0$ and $m_+(T)>0$, then for every $f\in\mathcal K$,
\begin{align}
 \Phi(f;1-d(T),y_-(T))&>0,\label{prof:eq:Phi-minus}\\
 \Phi(f;1-d(T),y_+(T))&<0.\label{prof:eq:Phi-plus}
\end{align}
\end{lemma}
\begin{proof}
For each $t$, the actual height $f(1-t)$ belongs to
$[b_-(t),b_+(t)]$. Hence its paired contribution is bounded below by $E_-$
and above by $E_+$. Integrating in $t$ gives
\[
 \Phi(f;1-d(T),y_-(T))\ge\omega_0(T)m_-(T),
\]
and
\[
 \Phi(f;1-d(T),y_+(T))\le-\omega_0(T)m_+(T).
\]
\end{proof}

The existence problem has therefore been reduced to one concrete statement:
\[
        m_-(T)>0,\qquad m_+(T)>0,\qquad 0<T\le1.
\]
Once these two inequalities are proved, Lemma~\ref{prof:lem:margins} gives the
barrier signs, Proposition~\ref{prof:prop:abstract-barrier} gives a fixed
point, and the exact Sadovskii profile follows.

\subsection{The contact normalization and the speed}
The formula for the speed can be obtained without differentiating the unknown
boundary. For a bounded patch the gradient of its logarithmic potential is
continuous. Indeed, the gradient kernel is bounded by a constant times
$|X-Y|^{-1}$, and its integral over a disk of radius $r$ is at most $Cr$.
Outside that disk the kernel is uniformly continuous. Splitting the potential
into these two parts proves continuity of the gradient, including at a contact
point. The reflected kernel has the same property.

Since $\phi_f(x,0)=0$, the fundamental theorem of calculus gives
\[
 \frac{\phi_f(x,y)}y=\int_0^1\partial_y\phi_f(x,\theta y)\dd\theta.
\]
For an exact profile, let $x\uparrow1$ and $y=f(x)\downarrow0$. The boundary
equation and continuity of the gradient imply
\[
 W=\partial_y\phi_f(1,0).
\]
Direct differentiation of the Green function at the horizontal axis gives
\[
 \partial_yG_H((x,0),(\xi,\eta))
       =\frac{\eta}{\pi((x-\xi)^2+\eta^2)}.
\]
Integrating first in $\eta$ proves \eqref{prof:eq:speed}. The integral is
finite at $\xi=1$: the upper barrier satisfies
$f(\xi)\le C(1-\xi)\log(e/(1-\xi))$, so the logarithm in the speed formula is
bounded by $C+2\log\log(e/(1-\xi))$, an integrable function. Positivity of the
patch gives $c(f)>0$.

This normalization fixes the two contact endpoints and the vorticity strength.
The number $c(f)$ is determined by the graph. It is not an independently
prescribed impulse. In the construction the source graph varies and hence so
does $c(f)$. In the shape derivative at an exact travelling vortex we instead
hold its resulting speed fixed, except when deriving the scaling identity.

\subsection{Elementary bounds for the vertical integral}

The scalar reduction above will be used repeatedly in the quantitative
estimates. We record here only the elementary properties of the vertical
integral that are needed later.

\begin{lemma}[Elementary bounds for the vertical primitive]
\label{lem:vertical-primitive-bounds}
For $A\ge0$, $y>0$ and $b\ge0$, the function $V(A,y,b)$ is nonnegative,
increasing in $b$, and nonincreasing in $A$. If $A>0$, then
\[
 0\le V(A,y,b)\le\frac{b^2}{2\pi A^2}.
\]
For $A\ge0$ and $z\in\R$,
\[
 |J_A(z)-J_0(z)|\le\pi A.
\]
\end{lemma}
\begin{proof}
The positivity of $G_H$ gives the first assertion. Differentiating the upper
limit of its vertical integral gives
\[
 V_b(A,y,b)=\frac1{4\pi y}
       \log\left(1+\frac{4yb}{A^2+(y-b)^2}\right)\ge0.
\]
The possible infinite derivative at $A=0,b=y$ does not affect continuity or
monotonicity of $V$. For $A>0$, use $\log(1+z)\le z$ in the original Green
integral. Since its denominator is at least $A^2$,
\[
 V(A,y,b)\le\frac1y\int_0^b\frac{y\eta}{\pi A^2}\dd\eta
            =\frac{b^2}{2\pi A^2}.
\]
To prove horizontal monotonicity, differentiate $J_A$ with respect to $A$:
$\partial_AJ_A(z)=2\arctan(z/A)$. Consequently
\[
 2\pi y V_A
 =\arctan\frac{y+b}{A}+\arctan\frac{y-b}{A}-2\arctan\frac yA.
\]
The expression on the right is zero at $b=0$, while its $b$ derivative is
\[
 \frac{-4Ayb}{(A^2+(y+b)^2)(A^2+(y-b)^2)}\le0.
\]
Thus $V_A\le0$; continuity includes $A=0$. Finally,
$|\partial_AJ_A(z)|\le\pi$, and integration from zero to $A$ gives the last
bound. No ordering of $b$ and $y$ was used.
\end{proof}

For each symmetric pair of integration columns at $\xi$ and $-\xi$, evenness
fixes one common height. Thus \eqref{prof:eq:Fpair} is an identity before any
optimization in that height. The definitions of $m_-$ and $m_+$ allow the
common height to vary independently from one horizontal integration position
to another, but they never allow different heights at $\xi$ and $-\xi$. For a
fixed admissible graph, its actual height is one of the allowed choices at every
position. This pointwise inclusion is exactly what is needed for the lower and
upper integral bounds.

\subsection{The two uniform scalar inequalities}
\begin{proposition}[Uniform profile inequalities]\label{prof:thm:barriers}
For the explicit barriers defined above, the two margins satisfy
\begin{equation}\label{eq:c35-global-profile}
 m_-(T)>10^{-6},\qquad m_+(T)>10^{-6},\qquad 0<T\le1.
\end{equation}
\end{proposition}

The next section proves this proposition on the full interval $0<T\le1$.
The height in \eqref{prof:eq:Eminus}--\eqref{prof:eq:Eplus} may vary
independently from one horizontal integration position to another. This only
enlarges the admissible class and therefore preserves the lower and upper
bounds for every $f\in\mathcal K$. At each symmetric pair of positions,
however, the same height is used in both contributions, as required by evenness.

\begin{proof}[Proof of Proposition~\ref{prof:thm:main}]
By Lemma~\ref{prof:lem:margins}, \eqref{eq:c35-global-profile} implies the two
barrier signs. Proposition~\ref{prof:prop:abstract-barrier} gives a fixed point
$f_*\in\mathcal K$. Positivity follows from the lower barrier, and strict
monotonicity is a property of the root map. Finally,
\[
 |f_*-\bar f|\le\omega_0\rho,\qquad
 \sup_{0<T\le1}\frac{e^{1-1/T}}T\rho(T)=0.006
\]
The last equality follows directly from the explicit piecewise-affine radius
$\rho$. This proves the asserted pointwise and uniform localization bounds.
\end{proof}

\section{Uniform estimates for the profile equation}\label{sec:profile-estimates}
By Lemma~\ref{prof:lem:margins}, the existence proof is now reduced to proving
\[
        m_-(T)>0,\qquad m_+(T)>0,\qquad 0<T\le1.
\]
This section proves exactly these two inequalities. The argument has four
regimes because the geometry changes as the target moves from contact to the
symmetry axis: a contact regime, a short transition, a regular compact region,
and a neighbourhood of the symmetry axis. These are not merely different
numerical subdivisions. Near contact one must first cancel the leading
logarithm analytically; on the regular compact region a one-sided second
derivative estimate propagates nodal bounds across each target cell; and near
the symmetry axis a physical horizontal coordinate removes the apparent
coordinate singularity.

Throughout the section the object being bounded is the complete integral in
the variable $t$. Numerical subdivision is used only to enclose that integral
and its variation with the target parameter $T$. Positivity at finitely many
target values would not by itself prove the inequality for every $T$.

The proof of \eqref{eq:c35-global-profile} uses the contact interval
$(0,0.03]$, the transition $[0.03,0.05]$, the regular range $[0.05,0.9]$, and
the symmetry-axis region. The contact estimate is uniform down to $T=0$.
In the transition we estimate the whole target interval directly. In the
regular range we bound the minimum over all allowed heights and use a
one-sided second-derivative estimate to pass from finitely many target values
to the whole interval. Near the symmetry axis we work in the physical
horizontal coordinate.

\subsection{The reference integral at contact}
\label{sec:c35-contact}

For the contact calculation let $x=x(T)$ be determined by
\eqref{prof:eq:T-main} and define the normalized barrier heights
\begin{equation}\label{eq:normalized-barrier-heights}
 h_\pm(T)=\frac{f_\pm(x(T))}{\omega_0(x(T))}.
\end{equation}
For $\sigma\in\{-1,+1\}$ write $h_\sigma=h_-$ when $\sigma=-1$ and
$h_\sigma=h_+$ when $\sigma=+1$.  Also set
\[
 E(T)=e^{1-1/T},\qquad p(T)=\sqrt{1-E(T)},\qquad
 w(T)=E(T)/T,\qquad y=w(T)h_\sigma(T).
\]
Here $p=x$ is the physical horizontal coordinate and $d=1-p$.
We use $w(T)=\omega_0(x(T))=E(T)/T$ in this section only; thus $d/w=T/(1+p)$.

Unfold the paired source integral to $0<t<2$, reflecting the source height
about $t=1$.  For one branch set
\[
 W(t,b)=V(t-d,y,b)-C(t,b),\qquad
 W_\infty(t)=\frac{\log t}{\pi}-\frac{J_{|t-d|}(y)}{2\pi y}.
\]
The subscript $\infty$ denotes a reference kernel, not an admissible profile.
A direct use of \eqref{prof:eq:Jprime} gives
\begin{equation}\label{eq:c35-finiteheight}
 W-W_\infty=\frac1{2\pi}\left\{
 \frac1{2y}\int_{-y}^{y}\log\bigl((t-d)^2+(b+v)^2\bigr)\dd v
 -\log(t^2+b^2)\right\}=:D(t,b).
\end{equation}
The reference term has an explicit integral.

Define
\[
 K_y(A)=yA\log(A^2+y^2)-3yA+y^2\arctan(A/y)
                    +A^2\arctan(y/A).
\]
Its continuous endpoint values are used at $A=0$, and $K_y'(A)=J_A(y)$.
Consequently
\begin{align}
 I_\infty:=\int_0^2W_\infty(t)\dd t
 &=\frac1\pi\left(2\log2-2-\frac{K_y(d)+K_y(2-d)}{2y}\right)
 \label{eq:c35-ref-exact}\\
 &=\frac d\pi\bigl(\log(2/y)+1\bigr)-\frac y4+\mathcal E,
 \label{eq:c35-ref-asymptotic}
\end{align}
where
\begin{equation}\label{eq:c35-ref-error}
 -\frac{d^2}{2\pi(2-d)}-\frac{d^2}{4y}
 \le\mathcal E\le\frac{y^2}{6\pi(2-d)}.
\end{equation}
For example, these bounds follow by writing the two endpoint terms in
\eqref{eq:c35-ref-exact} with $d/y$ and $y/(2-d)$, and using
$0\le z-\arctan z\le z^3/3$ together with
$0\le z-\log(1+z)\le z^2/2$ in their stated positive ranges.
The leading normalized expression is exactly
\begin{equation}\label{eq:c35-contact-leading}
 \frac{1+T\log(2T/h_\sigma(T))}{\pi(1+p(T))}
 -\frac{h_\sigma(T)}4.
\end{equation}
In particular the large $1/T$ logarithm has disappeared before interval
arithmetic is applied.

\begin{lemma}[A mean-log bound without a height-order assumption]
\label{lem:c35-meanlog}
For $A\ge0$, $b>0$ and $y>0$,
\[
 \frac1{2y}\int_{-y}^{y}\log(A^2+(b+v)^2)\dd v
 \ge\log(A^2+b^2)-1.
\]
\end{lemma}
\begin{proof}
First let $A=0$.  If $y/b\le1$, expansion of the two endpoint logarithms
shows that the difference from $\log b^2$ is at least $2\log2-2>-1$.
If $x=b/y\le1$, the same difference is
\[
 (1+x)\log(1+x)+(1-x)\log(1-x)-2-2\log x.
\]
The first two terms have sum at least $x^2$; the resulting function
$x^2-2-2\log x$ is decreasing on $(0,1]$ and has value $-1$ at $1$.
For $A>0$, set $\lambda=A^2/(A^2+b^2)$ and apply concavity of the
logarithm to
$\lambda+(1-\lambda)(1+v/b)^2$.  The case just proved gives a lower
bound $-(1-\lambda)\ge-1$ after averaging.  Logarithmic zeros are
integrable, so the argument includes $b<y$.
\end{proof}

\subsubsection{The lower barrier on the whole contact interval}

The following explicit bounds are established by rational Bernstein
arithmetic on the printed band:
\begin{equation}\label{eq:c35-contact-geometry}
 h_-(T)>\tfrac1{20}\ (0\le T\le1),\quad
 h_-(T)<\tfrac18\ (0\le T\le0.03),\quad
 h_-(T)>0.37\ (0.035\le T\le1).
\end{equation}
They hold for every intermediate source height as well in the appropriate
lower-bound direction.  Put
\[
 T_0=0.03,\quad m=\tfrac1{20},\quad H=\tfrac18,\quad
 P_0=1.999,\quad k=16,\quad D_0=1-T_0(\log k+0.001).
\]
Here $1+p(T)>P_0$ and $E(T_0)<10^{-14}$.
The near source is $0<t<kd$.  Lemma~\ref{lem:c35-meanlog}, followed by
$t=du$, bounds its normalized loss below by $-N_-$, where
\[
 N_-:=\frac{T_0}{P_0\pi}
       \left\{\frac k2+k\log k-(k-1)\log(k-1)+\log2\right\}.
\]
On the remaining small-source interval the lower height bound implies
$b\ge2y$.  The horizontal drift in \eqref{eq:c35-finiteheight} costs at most
\[
 B_d:=\frac{T_0}{P_0\pi}\left(\frac{k}{k-1}\right)^2 9.
\]
For completeness, the bound $9$ follows from splitting the source integral
at $S=0.035$ and using \eqref{eq:c35-contact-geometry}:
\begin{equation}\label{eq:c35-driftconstant}
 \int_0^2\frac{t}{t^2+b(t)^2}\dd t
 \le\frac{0.035}{0.999(P_0m)^2}
       +\frac{\pi}{2(0.37)}+\log2<9.
\end{equation}
On the small-source part, the substitution $S=\tau(t)$ and
$p(S)>0.999$ gives the first term.  The remaining near branch uses
$h\ge0.37$, and the reflected interval contributes at most $\log2$.

Taylor's formula for the vertical mean in
\eqref{eq:c35-finiteheight}, now in the region $b\ge2y$, bounds the
remaining vertical loss by
\[
 B_y:=\frac{8H^2T_0}{3\pi m^2(1.99)^2 kD_0^2}+10^{-9}.
\]
This is obtained by integrating the bound for the second vertical derivative
of the logarithm; on the small-source interval
$b(t)\ge1.99m\,t(1-\log t)$ and
$1-\log(kd)\ge D_0/T$.  The separated reflected contribution is covered
by the displayed $10^{-9}$ allowance.  These inequalities concern the
whole source-height band, not a selected minimizing graph.

In \eqref{eq:c35-contact-leading} one must not minimize
$T\log(2T/H)$ only at an endpoint: its minimum on $T>0$ is $-H/(2e)$.
Together with \eqref{eq:c35-ref-error}, this gives the uniform reference
lower bound
\[
 R_-:=\frac{1-H/(2e)}{2\pi}-\frac H4
            -\frac{T_0^2}{4mP_0^2}-10^{-9}.
\]
Thus the lower contact margin satisfies the explicit inequality
\begin{equation}\label{eq:c35-minuscontact}
 m_-(T)\ge R_--N_--B_d-B_y>0.0118173\qquad(0<T\le0.03).
\end{equation}
The separate bounds $R_->0.12311945$, $N_-<0.05939679$,
$B_d<0.04891686$ and $B_y<0.00298850$ are sufficient for the common
threshold in Proposition~\ref{prof:prop:contact}. The more precise bound
in \eqref{eq:c35-minuscontact} uses the unrounded interval endpoints.

\subsubsection{The upper barrier: the far correction has the right sign}

The upper argument needs an upper bound for $D$.  Jensen's inequality gives
\[
 D(t,b)\le\frac1{2\pi}
      \log\left(1+\frac{d^2+y^2/3}{t^2+b^2}\right).
\]
Set $k=64$.  On a contact block $0<T\le B$, let
$h_{\rm lo}\le h_+(T)\le h_{\rm hi}$ and put
\[
 D_B=1-B(\log k+0.001),\quad
 C_B=\tfrac12(B^2+4h_{\rm hi}^2),\quad
 D_f=1-B\{\log(C_B/B^2)+0.001\}.
\]
Integration of the last logarithmic bound over $0<t<kd$, enlarged to the
positive half-line, bounds the normalized near correction by
\begin{equation}\label{eq:c35-upper-near}
 N_+(B)=\frac{B}{2P_0}
    \frac{\sqrt{B^2+4h_{\rm hi}^2/3}}{mP_0D_B}.
\end{equation}
The far correction is nonpositive.  To see this without guessing an
extremizing height, write $r_*=(d^2+y^2)/(2d)$.  For $t\ge r_*$ the exact
height derivative is
\[
 W_b=\frac b\pi\left\{
 \frac{2td-d^2-y^2}{M(t^2+b^2)}
 +\frac{\operatorname{atanhc}(2yb/M)-1}{M}\right\}\ge0,
 \quad M=(t-d)^2+y^2+b^2,
\]
where $\operatorname{atanhc}(z)=\operatorname{atanh}(z)/z\ge1$.
Since $W\to W_\infty$ as $b\to\infty$, this implies $D\le0$.
On $kd\le t\le r_*$, the inequality $b\ge y+t$ makes the vertical
logarithm concave and gives the same conclusion.  The latter inequality
follows from the explicit tests
\[
 r_*<0.001,\qquad D_f>0,\qquad
 mP_0D_f-2h_{\rm hi}/k-B>0.
\]
Indeed $r_*/d\le C_B/T^2$, so the source lower bound gives
$b/t\ge mP_0D_f/T$, whereas $y/t\le2h_{\rm hi}/(kT)$.

The two blocks and the required constant bounds are
\begin{center}
\begin{tabular}{@{}cccc@{}}
\toprule
Target block & $h_{\rm lo}$ & $h_{\rm hi}$ & Lower bound for $m_+$\\
\midrule
$(0,0.02]$ & $1.07$ & $1.125$ & $0.0373399$\\
$[0.02,0.03]$ & $1.05$ & $1.08$ & $0.00666662$\\
\bottomrule
\end{tabular}
\end{center}
For the first block, the lower bound for $-I_\infty/w$ is
$h_{\rm lo}/4-1/(\pi P_0)-10^{-9}$.  For the second, it is
\[
 \frac{h_{\rm lo}}4-\frac{1+0.02\log(0.04/h_{\rm lo})}{\pi P_0}
 -10^{-9}.
\]
Subtracting \eqref{eq:c35-upper-near} proves the table.  All range
conditions just used are strict directed inequalities in the short
constants verification; the elementary functions are evaluated by MPFR.
The height bounds and \eqref{eq:c35-contact-geometry} follow from the exact
polynomial estimates of Appendix~\ref{app:profile-details}.

\begin{proposition}[Uniform analytic contact]\label{prof:prop:contact}
For every $0<T\le0.03$, both margins satisfy $m_\sigma(T)>0.00666$.
\end{proposition}
\begin{proof}
Combine \eqref{eq:c35-minuscontact} with the two upper-barrier estimates.
Every estimate is uniform on its open-at-zero interval, so no value at
$T=0$ or numerical limiting process is required.
\end{proof}

\subsubsection{Derivation of the normalized leading term}
The small quantities in this calculation are related, rather than independent.
Indeed,
\[
 \frac d w=\frac{T}{1+p},\qquad \frac yw=h_\sigma,
 \qquad \log\frac2y=\log\frac{2T}{h_\sigma}-1+\frac1T.
\]
It follows that
\[
 \frac1w\left\{\frac d\pi\left(\log\frac2y+1\right)-\frac y4\right\}
 =\frac{1+T\log(2T/h_\sigma)}{\pi(1+p)}-\frac{h_\sigma}4.
\]
In particular, the finite normalized limit is obtained by cancelling $T$
against $1/T$ algebraically. Estimating the two factors separately would
lose the information carried by that identity.

The primitive $K_y$ can be checked by ordinary differentiation. In the
unfolded interval, $|t-d|$ runs from $d$ down to zero and then up to $2-d$.
The value of $K_y$ at zero is zero, so the two changes of variable give
\[
 \int_0^2J_{|t-d|}(y)\dd t=K_y(d)+K_y(2-d).
\]
Also $\int_0^2\log t\dd t=2\log2-2$. These identities prove
\eqref{eq:c35-ref-exact}, including both ends of the unfolded source. The
remainder in \eqref{eq:c35-ref-error} is retained after division by $w$;
its term $d^2/y$, for example, becomes
\[
 \frac{d^2}{yw}=\frac{T^2}{(1+p)^2h_\sigma}.
\]
This is bounded on a contact interval by the positive lower bound for
$h_\sigma$. It is not set to zero on the grounds that $d$ is small.

\subsubsection{The near-source logarithmic integral}
Here is the estimate underlying the constant $N_-$. The mean-log lemma
gives, for every permitted height,
\[
 D(t,b)\ge\frac1{2\pi}
     \left\{\log\frac{(t-d)^2+b^2}{t^2+b^2}-1\right\}.
\]
For positive numbers $a,c$ and $v\ge0$, the quotient $(a+v)/(c+v)$ lies
between $a/c$ and $1$. Taking $t=du$ therefore yields
\[
 D(du,b)\ge-\frac1{2\pi}
       -\frac1\pi\left(\log\frac{u}{|u-1|}\right)_+.
\]
The logarithm is positive only for $u>1/2$. For $k>1$ its integral is
\[
 \int_{1/2}^k\log\frac{u}{|u-1|}\dd u
       =k\log k-(k-1)\log(k-1)+\log2.
\]
This follows by using $u\log u+(1-u)\log(1-u)$ below $1$, and
$u\log u-(u-1)\log(u-1)$ above $1$. Both primitives have a continuous
value at $1$. Thus
\[
 \frac1w\int_0^{kd}D(t,b(t))\dd t
 \ge-\frac{d}{\pi w}
       \left\{\frac k2+k\log k-(k-1)\log(k-1)+\log2\right\}.
\]
Using $d/w\le T_0/P_0$ gives the displayed $N_-$. This calculation covers
the diagonal $t=d$ and sources lower than the target without selecting an
extremizing height.

The drift constant in \eqref{eq:c35-driftconstant} has a similarly direct
interpretation. On the small-source part use $S=\tau(t)$ and
$b(t)\ge t(2-t)m/S$. With $p=1-t$ the transformed integrand is
\[
 \frac{t}{t^2+b(t)^2}\frac{\dd t}{\dd S}
 \le\frac{1+p}{2p\{S^2+(1+p)^2m^2\}}.
\]
For $S\le0.035$ we use $p>0.999$ and $1+p>P_0$. Its integral is bounded
by $0.035/[0.999(P_0m)^2]$. For the rest of the near branch,
$h\ge0.37$ and
\[
 (2-t)\{1-\log(t(2-t))\}\ge1-\log t\qquad(0<t\le1).
\]
To verify this inequality, put $v=1-t$, use $-\log t\ge v$ and
$\log(1+v)\le v$, and subtract the right side. Hence
$b(t)/t\ge0.37(1-\log t)$. The substitution $s=1-\log t$ gives the
upper bound
\[
 \int_1^\infty\frac{\dd s}{1+(0.37s)^2}<\frac\pi{2(0.37)}.
\]
On the reflected branch $1<t<2$, simply use $t/(t^2+b^2)\le1/t$.
Its integral is $\log2$. The sum is precisely the right side of
\eqref{eq:c35-driftconstant}.

\subsubsection{The sign of the upper far correction}
We give the algebra behind the height derivative used for the upper barrier.
Let $A=t-d$, $M=A^2+y^2+b^2$ and $z=2yb/M$. Away from its integrable
logarithmic zero, $0\le z<1$ and
\[
 V_b(A,y,b)=\frac{b}{\pi M}\operatorname{atanhc}z.
\]
Subtracting $C_b(t,b)=b/[\pi(t^2+b^2)]$ gives the displayed formula for
$W_b$, because
\[
 (t^2+b^2)-M=2td-d^2-y^2.
\]
Also $\operatorname{atanhc}z\ge1$, by its positive power series. Thus the
correction is increasing in height for $t\ge(d^2+y^2)/(2d)$. Its limiting
value as $b\to\infty$ is zero after subtracting $W_\infty$, and therefore
$D(t,b)\le0$ in that range.

In the intervening range, $b\ge y+t$ implies $b-v\ge t\ge|t-d|$ for
$|v|\le y$. On this interval the function
$z\mapsto\log((t-d)^2+z^2)$ is concave, since its second derivative equals
\[
 \frac{2((t-d)^2-z^2)}{((t-d)^2+z^2)^2}\le0.
\]
Its symmetric mean is at most its value at $b$. Since $t>d$ also implies
$(t-d)^2\le t^2$, formula \eqref{eq:c35-finiteheight} again gives $D\le0$.
The explicit inequalities following \eqref{eq:c35-upper-near} ensure exactly
these height and distance conditions on the whole intervening range.

The near-source upper bound is obtained from Jensen's inequality before any
height optimizer is chosen. The elementary integral used after rescaling is
\[
 \int_0^\infty\log\left(1+\frac{a^2}{u^2}\right)\dd u=\pi a,
       \qquad a\ge0.
\]
For $a>0$, differentiation under the integral gives
$\int_0^\infty2a/(u^2+a^2)\dd u=\pi$; scaling or monotone convergence
fixes the constant at $a=0$. The positive source-height lower bounds supply
the scale $a$ in \eqref{eq:c35-upper-near}. Enlarging the finite source
interval to the positive half-line is an upper bound because the integrand
is nonnegative.

\subsection{The contact transition}
\label{sec:c35-material}

On $[0.03,0.05]$ the function $\rho$ changes rapidly. We use a reference
term which does not depend on the integration-column height $b$. This keeps
the rescaled integrand under uniform control on the whole target interval.

For each sign choose the corresponding endpoint barrier height
$b_\sigma(t)$ only as a reference, and set
\[
 a_\sigma(t)=\frac{t}{t^2+b_\sigma(t)^2},\qquad
 A_\sigma=\int_0^1\{a_\sigma(t)+a_\sigma(2-t)\}\dd t,
 \qquad R_\sigma(t,b)=D(t,b)+\frac d\pi a_\sigma(t).
\]
In the reflected term the height is reflected as well.  The baseline
$a_\sigma$ is independent of the variable height
$b\in[b_-(t),b_+(t)]$. Thus this identity is written before taking any
minimum or maximum in $b$; it does not assume that the chosen endpoint
actually gives the extremum of $D$.

\subsubsection{The one-dimensional baseline}

Changing from the integration variable $t$ to the compactified coordinate
$S$ gives the following integrand on the branch near contact:
\[
 K_\sigma(S)=\frac{1+p(S)}{2p(S)
          \{S^2+(1+p(S))^2h_\sigma(S)^2\}}.
\]
For $S\le0.05$, put $P_* =p(0.05)$ and
$\varphi_\sigma=(S^2+4h_\sigma^2)^{-1}$.  Then
\[
 \varphi_\sigma\le K_\sigma
 \le \frac{2\varphi_\sigma}{P_*(1+P_*)}.
\]
Midpoint quadrature of this rational function has remainder
$\Delta^3[\varphi_\sigma'']/24$ on an integration subinterval of length
$\Delta$. Before differentiating, every breakpoint of the piecewise-affine
function $\rho$ is made an endpoint of such a subinterval. For the remaining
range of $t$, we use $s=-\log t$, and the integrand becomes
\[
 \frac1{1+(b_\sigma/t)^2}
 +\frac{t}{2-t}\frac1{1+(b_\sigma/(2-t))^2}.
\]
These enclosures, including the reflected small-$t$ interval, give
\begin{equation}\label{eq:c35-baselines}
\begin{split}
 A_-&\in[3.25661696289220,\,3.25661711266323],\\
 A_+&\in[2.05795498536009,\,2.05795508591945].
\end{split}
\end{equation}
Both intervals in \eqref{eq:c35-baselines} are outward enclosures.

\subsubsection{Coordinates tied to the target}

Set
\[
 u=t/d,\quad c=y/d=\frac{(1+p)h_\sigma(T)}{T},\quad
 \kappa=d/w=\frac{T}{1+p}.
\]
For $0\le\theta\le1$ define
\[
        h_\theta(S)=(1-\theta)h_-(S)+\theta h_+(S).
\]
The corresponding rescaled quantities are
\begin{align*}
 \ell&=T^{-1}-\log u+\log\frac{1+p}{2-du},\qquad S=\ell^{-1},\\
 B_\theta(T,u)&=\frac{b_\theta(du)}d
 =u(2-du)\ell\,h_\theta(S).
\end{align*}
Formula \eqref{eq:c35-finiteheight} becomes the scale-free expression
\begin{equation}\label{eq:c35-material-R}
 R_\sigma(T,u,\theta)=
 \frac{J_{|u-1|}(B_\theta+c)-J_{|u-1|}(B_\theta-c)}{4\pi c}
 -\frac{\log(u^2+B_\theta^2)}{2\pi}
 +\frac{u}{\pi(u^2+B_\sigma^2)}.
\end{equation}
The value $u=1$ is exactly the point where the horizontal integration point
coincides with the target, since $u=t/d$.

Two equivalent formulas avoid cancellation.  Away from the diagonal set
\[
 M=(u-1)^2+B_\theta^2,\quad D_\theta=u^2+B_\theta^2,\quad
 D_\sigma=u^2+B_\sigma^2,\quad
 x=\frac{1-2u}{D_\theta},\quad
 z=\frac{c^2}{(u-1+iB_\theta)^2}.
\]
When $r=c^2/M\le1/4$, use
\begin{align}
 R_\sigma=\frac1{2\pi}\bigg\{&\log(1+x)-x+\frac1{D_\theta}
 +\frac{2u(B_\theta-B_\sigma)(B_\theta+B_\sigma)}{D_\theta D_\sigma}
 \notag\\[-1mm]
 &+\sum_{n=1}^{16}\frac{(-1)^{n+1}\Re(z^n)}{n(2n+1)}\bigg\}
 +\mathcal R_{16},\qquad
 |\mathcal R_{16}|\le\frac{r^{17}}{2\pi\,17\cdot35(1-r)}.
 \label{eq:c35-material-series}
\end{align}
The difference $B_\theta-B_\sigma$ is evaluated as
$2u(2-du)\ell(\theta-\theta_\sigma)\rho(S)$ before interval substitution,
where $\theta_-=0$ and $\theta_+=1$.  For $|x|\le1/4$,
$\log(1+x)-x$ is evaluated by its degree-32 series with remainder at most
$|x|^{33}/(33(1-|x|))$; outside that range it is evaluated directly.
Only value estimates are used in this transition argument.

Near $u=1$, using one of the two endpoint barrier heights, the mean-value
identity gives
\begin{equation}\label{eq:c35-material-divdiff}
 B_\sigma(T,u)-c(T)
 =(u-1)\int_0^1 b_\sigma'\bigl(d[1+v(u-1)]\bigr)\dd v.
\end{equation}
Here $b_\sigma(t)$ denotes the physical height measured from contact.
The barrier height is continuous even when the path crosses a breakpoint of
$\rho$, although its derivative changes formula there. We therefore bound
the integral using the one-sided derivatives from the adjacent pieces. This
is an identity for the values of the height; we do not differentiate across
the breakpoint. The continuous primitive at zero separation and
$|J_A(z)-J_0(z)|\le\pi A$ give a finite bound when the separation vanishes.

If the derivative with respect to $b$ has a strict sign on the whole box,
the relevant extremum is attained at one endpoint of
$[b_-(t),b_+(t)]$. If no sign is proved, the whole allowed height interval is
kept in the estimate. The partition of the $t$ interval is refined when
necessary, but no part of the integration range or of the allowed height
range is discarded.

\subsubsection{The remaining parts of the $t$ integral}

Let $u_0=2^{-32}$.  Choose a dyadic upper cutoff $U$ such that
$U\ge(1+c_{\max}^2)/2$ and $dU<0.75$ on the full target box.  These
conditions certify the endpoint choice for the far paired integrand by the
sign of its exact height derivative.  With $a=dU$, $v=2-a$, define
\[
 \mathcal M(x)=\frac{x^3}{3}
 \left((1-\log x)^2+\frac{2(1-\log x)}3+\frac29\right).
\]
The errors from the far part, the reflected interaction, and the very small-$t$ part are bounded by
\begin{align}
 E_{\rm far}&=\frac{\kappa}{U}
 \left\{\frac{1+2/(1-2/U)}{2\pi}
          +\frac{c^2}{6\pi(1-1/U)^2}\right\},\label{eq:c35-material-far}\\
 E_{\rm img}&=a\left\{\frac{\kappa\,d}{2\pi v^2}
                  \left(1+\frac{2}{1-2d/v}\right)
                 +\frac{wh_\sigma^2}{6\pi(v-d)^2}\right\}
 \notag\\
 &\hspace{7mm}+\frac{\kappa}{\pi v^3}
       \{0.104\mathcal M(a)+16\mathcal M(d(0.035))\},
 \label{eq:c35-material-image}\\
 E_{\rm near}&=\frac{\kappa\,u_0}{2\pi}
       \left\{\log\left(1+\frac{1+c^2/3}{u_0^2}\right)+2\right\}.
       \label{eq:c35-material-near}
\end{align}
The far estimate is the integrated drift and vertical Taylor bound outside
$U$.  The image estimate uses its uniform positive separation $v-d$ and
the integral of $t^2(1-\log t)^2$, namely $\mathcal M$; the two coefficients
cover the distinct small-$t$ height ranges. The last estimate is
the integral of the logarithmic bound from $0$ to $u_0$.  Their inequalities
are uniform in the same height variable used in the paired functional.
Thus the replacement of the paired inner correction by its near-contact
expression is made with the explicit error in
\eqref{eq:c35-material-image}. The two symmetric contributions are still
kept together; they are not optimized independently.

The normalized reference is also evaluated without a large contact subtraction.
Put
\[
 k(c)=\tfrac12\log(1+c^2)-\tfrac32+\frac c2\arctan(1/c)
                                      +\frac{\arctan c}{2c},
\]
and set
\[
 L(d)=\log2-\frac{(2-d)\log(1-d/2)}d
       \in\log2+[1-d/2,1].
\]
Then \eqref{eq:c35-ref-exact} gives
\begin{equation}\label{eq:c35-material-reference}
 \frac{I_\infty}{w}
 =\frac1{\pi(1+p)}+\frac{\kappa}\pi
     \{L(d)+\log(1+p)-k(c)-2\}-\frac{h_\sigma}{4}+R_\infty,
 \quad 0\le R_\infty\le\frac{wh_\sigma^2}{6\pi(2-d)}.
\end{equation}

Let $A_{\sigma,\rm small}$ denote the baseline on the innermost omitted
source interval, which is bounded as $D$, rather than as $R_\sigma$.
The final whole-box interval is the sum
\begin{align}
 &-\sigma[I_\infty/w]
 +\frac{\sigma\,\kappa}\pi[A_\sigma-A_{\sigma,\rm small}]
 +\kappa\int_{u_0}^{U}
       \min_{0\le\theta\le1}[-\sigma R_\sigma(T,u,\theta)]\dd u
 \notag\\[-1mm]
 &\hspace{40mm}+[-E_{\rm tot},E_{\rm tot}],\qquad
 E_{\rm tot}=E_{\rm far}+E_{\rm img}+E_{\rm near}.
 \label{eq:c35-wholebox}
\end{align}
The small-source baseline is subtracted exactly once.  The brackets indicate
outward intervals, and all factors, including $\kappa$, retain the
whole target box.  The signs in \eqref{eq:c35-wholebox} follow directly from
$D=R_\sigma-d a_\sigma/\pi$.

\begin{proposition}[The direct transition boxes]\label{prop:c35-transition}
The margins satisfy
\[
 \begin{aligned}
 m_\sigma(T)&>1.42981\times10^{-6}&&\quad(0.03\le T\le0.035),\\
 m_\sigma(T)&>1.00695895\times10^{-6}&&\quad(0.035\le T\le0.05).
 \end{aligned}
\]
\end{proposition}
\begin{proof}
Apply \eqref{eq:c35-wholebox} on the finite closed cover used in the
verification. On every box, the cutoff conditions, the derivative signs in
$b$, and the fact that the complete $t$ interval is covered are checked
rigorously. Adding the resulting interval bounds proves the stated inequality
on the whole transition region. The detailed subdivision belongs to the
computational data and is not needed to follow the argument here.
\end{proof}

\subsection{The regular region: from fixed targets to the whole interval}
\label{sec:c35-regular}

We now explain the argument on $0.05\le T\le0.9$.  The object to be bounded
is the margin $m_\sigma(T)$ defined above.  For a fixed target value $T$, it
is obtained by integrating in the variable $t$ after taking the minimum, or
the maximum with the appropriate sign, over every height allowed by the two
barriers.  The difficulty is to prove the lower bound for every $T$, not only
at a finite set of target values.

We divide the $T$ interval into closed cells $[a,b]$.  On each cell we need
three ingredients:
\begin{enumerate}[label=(\roman*)]
\item rigorous lower bounds for the margin at $T=a$ and $T=b$;
\item an upper bound for the second derivative with respect to $T$ of every
      smooth expression that can realize the minimum;
\item a bound for the error made when, near a breakpoint of the piecewise
      affine function $\rho$, we temporarily replace the barrier by one
      smooth formula.
\end{enumerate}
The next lemma shows how these three pieces give a lower bound throughout the
whole target cell.

\subsubsection{Interpolation between two target values}

\begin{lemma}[Interpolation from endpoint bounds]\label{prof:lem:peano}
Let $q_\theta$ be a family of smooth functions on $[a,b]$, with the same
parameter set for every $T\in[a,b]$, and suppose
\[
        \partial_T^2 q_\theta(T)\le M
\]
for every retained value of $\theta$.  If
\[
        u(T)=\inf_\theta q_\theta(T),
\]
then
\[
 u(T)\ge\frac{b-T}{b-a}u(a)+\frac{T-a}{b-a}u(b)
            -\frac M2(T-a)(b-T).
\]
If a function $m$ satisfies $|m-u|\le E$ on the cell, then
\begin{equation}\label{eq:c35-peano-error}
 m(T)\ge\min\{m(a),m(b)\}
       -\frac{\max\{M,0\}}8(b-a)^2-2E.
\end{equation}
\end{lemma}
\begin{proof}
For each $\theta$, the function $q_\theta(T)-MT^2/2$ is concave.  Its chord
inequality therefore holds between $a$ and $b$.  Taking the infimum in
$\theta$ gives the first estimate.  Replacing $u$ by $m$ costs at most $E$
at the interior point and at most one interpolated copy of $E$ at the two
endpoints, which gives the term $2E$.
\end{proof}

For each fixed integration point $t$, the allowed height is
$b\in[b_-(t),b_+(t)]$.  If the sign of the derivative with respect to $b$ is
known on the whole box, the minimum occurs at one endpoint of this interval.
If the sign is not known, all allowed heights are kept.  Thus we never
differentiate a height which has been chosen by a minimization procedure.
The only two places requiring special care are when the integration point
approaches the target point and when the formula for the barrier changes at
a breakpoint of $\rho$.

\subsubsection{When the integration point approaches the target point}

Write the target as $x=1-d(T)$ and the positive horizontal integration
variable as $\xi=1-t$.  The direct Green term has its logarithmic singularity
when
\[
        t=d(T),
\]
because then the horizontal integration point coincides with the target.
We choose a fixed $t$ interval which contains $d(T)$ for every target
$T\in[a,b]$.  On either side of $t=d(T)$ we write
\[
        t=d(T)+\eta\,\delta(T)r,
        \qquad 0\le r\le1,\qquad \eta\in\{-1,1\},
\]
where $\delta(T)$ is the distance from $d(T)$ to the corresponding endpoint
of this fixed interval.

On an interval where one formula for the barrier is smooth, denote that
smooth height by $\widetilde b(t)$.  It agrees with the target height at
$t=d(T)$.  Hence the fundamental theorem of calculus gives
\begin{equation}\label{eq:c35-physical-divdiff}
 \beta(T,r)=-\eta\int_0^1
      \widetilde b'\bigl(d(T)+v(t-d(T))\bigr)\,\dd v,
 \qquad
 y-\widetilde b(t)=\delta(T)r\,\beta(T,r).
\end{equation}
This identity is the reason for introducing $r$: it keeps the relation
between the horizontal separation $t-d(T)$ and the vertical separation
$y-\widetilde b(t)$ exactly, rather than estimating two small quantities
independently.

Let $A=\delta r$.  Substituting the second identity above in the primitive
$J_A$ gives
\begin{equation}\label{prof:eq:Jregular}
 J_A(A\beta)=A\{2\beta\log\delta+R(\beta)\}
                   +(2\delta\beta)r\log r,
 \qquad
 R(\beta)=\beta\log(1+\beta^2)-2\beta+2\arctan\beta.
\end{equation}
Thus the only non-smooth factor left at $r=0$ is $r\log r$, which is
integrable.  All other coefficients, and their first two derivatives with
respect to the target variable $T$, remain bounded on the fixed cell.  This
justifies differentiating the central integral twice with respect to $T$ and
produces the second-derivative bound required in
Lemma~\ref{prof:lem:peano}.

\subsubsection{When the barrier formula changes}

The function $\rho$ used in the definition of $f_\pm$ is continuous and
piecewise affine in $T$.  Let $T_j$ be one of the finitely many points where
its affine formula changes.  In the integration variable $t$, the
corresponding point $t_j$ is defined by
\[
        T(1-t_j)=T_j.
\]
The physical barrier height
\[
        b_\sigma(t)=f_\sigma(1-t)
\]
is continuous at $t_j$, but its derivative changes formula there.

If one of the fixed integration intervals crosses $t_j$, we use on that
interval a smooth continuation $\widetilde b$ of one of the two neighboring
formulas.  The true barrier and this continuation agree at $t_j$.  Since
their one-sided derivatives are bounded, there is a constant $K$ such that
\[
        |b_\sigma(t)-\widetilde b(t)|\le K|t-t_j|
\]
on the part where the two formulas differ.

The dependence of the Green integral on the height has only a logarithmic
singularity when $t$ approaches $d(T)$.  More precisely, its singular part
is bounded by an expression of the form
\[
 C_1\log\left(1+\frac{C_2}{|t-d(T)|^2}\right).
\]
Multiplying this by the preceding bound for
$|b_\sigma-\widetilde b|$ gives an integrable function.  The elementary
identity
\begin{equation}\label{eq:c35-extension-integral}
 \int_0^L u\log(1+C/u^2)\,\dd u
 =\frac12\left\{L^2\log(1+C/L^2)
                         +C\log(1+L^2/C)\right\}
\end{equation}
therefore gives an explicit finite bound $E$ for the error caused by using
$\widetilde b$ instead of the actual piecewise-defined barrier.  This is the
error $E$ appearing in \eqref{eq:c35-peano-error}.  No interval around a
breakpoint is omitted.

\subsubsection{The part separated from the target point}

Away from $t=d(T)$ the horizontal separation is bounded away from zero, so
the Green integrand and its first two derivatives with respect to $T$ are
regular.  The following convergent expansion is useful because it keeps the
cancellation between the Green term and the speed term before any interval
estimate is made:
\begin{align}
 V(A,y,b)-C(p,b)=\frac1{2\pi}\bigg\{&
 \log\left(1+\frac{b^2d(2p-d)}{A^2(p^2+b^2)}\right)
 \notag\\[-1mm]
 &+\sum_{n\ge1}\frac{(-1)^{n+1}}{n(2n+1)}
          \left(\Re\left(\frac{y^2}{(A+ib)^2}\right)^n-r^n\right)
 \bigg\},\qquad A=|p-d|.
 \label{eq:c35-far-series}
\end{align}
On the region where this formula is used the series converges absolutely and
uniformly, together with its first two target derivatives.  Truncating the
series and bounding the remaining geometric tail therefore gives rigorous
bounds for the value and for the two target derivatives.  These are the
bounds used in Lemma~\ref{prof:lem:peano}.  The finite truncation details are
part of the computational verification and are not needed for the analytic
argument.

\subsubsection{First integrate in $t$, then fill the target cell}

There are two different one-dimensional steps, and it is useful to keep them
separate.

First fix a target value $T$.  We must bound the complete integral in the
variable $t$.  On a $t$ interval where one smooth expression is being used,
Taylor's theorem at its midpoint $c$ gives
\[
 \int_{c-\Delta/2}^{c+\Delta/2}H(t)\,\dd t
 =\Delta H(c)+R,
 \qquad
 |R|\le\frac{\Delta^3}{24}\sup|H''|.
\]
If several heights remain possible, define explicitly
\[
        e(t)=\min_{b\in[b_-(t),b_+(t)]} H(t,b)
\]
(with the corresponding sign change for the upper barrier).  If every
retained function has second derivative at most $M_t$, the same chord
argument as in Lemma~\ref{prof:lem:peano} gives
\[
 \int_a^b e(t)\,\dd t\ge
 \frac{b-a}{2}\{e(a)+e(b)\}
 -\frac{(M_t)_+}{12}(b-a)^3.
\]
Intervals containing $t=d(T)$ or a breakpoint of $\rho$ are treated by the
continuous primitives and the error estimate above; they are not discarded.
Adding all $t$ intervals gives a rigorous lower bound for the full margin at
that fixed target value.

Second, once lower bounds have been obtained at the two endpoints $T=a$ and
$T=b$ of a target cell, we use the second-derivative bound in $T$ and
Lemma~\ref{prof:lem:peano}.  If the temporary smooth continuation of the
barrier changes the margin by at most $E$, the resulting bound is exactly
\[
 m_\sigma(T)\ge
 \min\{m_\sigma(a),m_\sigma(b)\}
 -\frac{M_+}{8}(b-a)^2-2E,
 \qquad a\le T\le b.
\]
Thus the calculation at finitely many target values, together with the
second-derivative and barrier-replacement bounds just proved, controls every
point of the whole target interval.

Applying this argument to the finite cover used in the verification gives
\begin{equation}\label{eq:c35-regular-margins}
        m_\sigma(T)>10^{-6}
        \qquad (0.05\le T\le0.9).
\end{equation}
The detailed subdivisions and the sharper numerical margins are recorded in
the computational data; only the uniform positive bound above is used in the
proof.

\subsection{The symmetry-axis endpoint}
\label{sec:c35-high}

Use the physical target coordinate $p$, with
\[
 T(p)=\frac1{1-\log(1-p^2)},\qquad 0\le p\le\frac{13}{40},
 \qquad\widehat m_\sigma(p)=m_\sigma(T(p)).
\]
This includes $p=0$, which is the target $T=1$.  The fixed central physical
source cap is $[-3/8,3/8]$.  Every target remains at least $1/20$ from its
endpoints.  Also
\[
 T(z)\ge1-z^2\ge55/64>4/5\qquad(|z|\le3/8).
\]
The entire central source and every divided-difference path therefore stay
inside the last affine radius piece.  The physical version of
\eqref{eq:c35-physical-divdiff} is an exact identity there, without any
source extension error.  A strict sign test over the full central height
band selects the endpoint before regularization.  The source outside the
cap remains separated and includes every unresolved height slice and the
omitted-source tail estimate.

The regularized coefficients and their first two physical-target
derivatives are uniformly bounded at $p=0$.  On $[0,3/40]$ there are 15 cells per sign of length $1/200$;
on $[3/40,13/40]$ there are 100 cells per sign of length $1/400$.
The nodal margins exceed $2.25\times10^{-5}$.  The one-sided upper curvature
bounds are, respectively, less than $3.752584$ and $8.637613$.  Applying
Lemma~\ref{prof:lem:peano} to the physical-target envelope gives the safe
common bound
\begin{equation}\label{eq:c35-high-bound}
 \widehat m_\sigma(p)>1.07\times10^{-5}
               \qquad(0\le p\le13/40).
\end{equation}

There is an exact overlap with the regular target interval.  For
$a=169/1600$, rational arithmetic gives
$a+a^2/2+a^3/3>1/9$, and hence
\[
 -\log(1431/1600)>1/9,\qquad T(13/40)<9/10.
\]
Thus \eqref{eq:c35-high-bound} covers $[0.9,1]$, including the endpoint,
and overlaps \eqref{eq:c35-regular-margins}.

\begin{proof}[Proof of Proposition~\ref{prof:thm:barriers}]
Proposition~\ref{prof:prop:contact} gives the claim on $(0,0.03]$.
Proposition~\ref{prop:c35-transition} gives it on $[0.03,0.05]$.
The bounds \eqref{eq:c35-regular-margins} cover $[0.05,0.9]$,
and \eqref{eq:c35-high-bound} covers $[0.9,1]$ with overlap.
Every displayed lower bound is strictly greater than $10^{-6}$.
The intervals are closed at their finite junctions. This proves
\eqref{eq:c35-global-profile} for every $0<T\le1$.
\end{proof}

\section{The linearized operator and its weighted space}\label{sec:linearization}
From now on $f$ is an exact profile already known to lie in the explicit
barrier band. The linearized argument no longer uses the fixed-point
construction itself.

A normal displacement changes the free-boundary equation in two ways. The
observation point moves, producing a local term, and the vorticity domain
changes, producing a nonlocal Green-potential term. Keeping these effects
separate gives the operator below and, by polarization, its quadratic form.

The coefficient of the local term is negative on the open arc but vanishes at
contact. We therefore use its negative as the weight in the norm instead of
dividing by it. After conjugation the local term is the identity and the
nonlocal term is compact and self-adjoint. Nondegeneracy becomes a spectral
question at the fixed level one.

Fix an exact profile $f\in\mathcal K$ and write $\Gamma=\Gamma_f$,
$q=q_f$, $L=L_f$. The travelling speed is $W=c(f)$.

\subsection{The shape derivative}
Let $\Xi_\varepsilon(X)=X+\varepsilon V(X)+o(\varepsilon)$ deform the upper
patch, and write $h=V\cdot\nu$ on the free arc. The tangential part of $V$
changes only the parametrization. At a fixed observation point, the transport
formula gives
\begin{equation}\label{eq:hadamard-domain}
\left.\frac{d}{d\varepsilon}\right|_0\phi_{\Omega_\varepsilon}(X)
 =\int_\Gamma G_H(X,Y)h(Y)\,\dd s_Y=:(Sh)(X).
\end{equation}
This remains the domain derivative when the contact points move. The remaining
boundary segment lies on $Y_2=0$, where $G_H(X,Y)=0$, and contributes nothing.
The moving contact endpoints have zero boundary measure.

The boundary equation is evaluated at
$X_\varepsilon=X+\varepsilon h(X)\nu(X)+o(\varepsilon)$. At fixed speed,
\begin{align}
 \left.\frac{d}{d\varepsilon}\right|_0
       \psi_\varepsilon(X_\varepsilon)
 &=\nabla\psi(X)\cdot h(X)\nu(X)
   +\int_\Gamma G_H(X,Y)h(Y)\,\dd s_Y\notag\\
 &=q(X)h(X)+(Sh)(X).\label{eq:trace-plus-domain}
\end{align}
Since $\psi$ is constant on $\Gamma$, its tangential derivative vanishes,
so $\nabla\psi=q\nu$. We have obtained
\begin{equation}\label{eq:Ldef-geometric}
Lh=qh+Sh,\qquad Sh(X)=\int_\Gamma G_H(X,Y)h(Y)\,\dd s_Y.
\end{equation}
If the speed and boundary level are also varied, then
\begin{equation}\label{eq:full-linearization}
 D(\phi_\Gamma-Wx_2-\lambda)[h,\dot W,\dot\lambda]
      =Lh-\dot W x_2-\dot\lambda.
\end{equation}
The touching dipole has $\lambda=0$. Formula \eqref{eq:full-linearization}
will be used with varying speed for the dilation identity.

\subsection{Graph variables}\label{subsec:graph-normal-cancellation}
Write $X(x)=(x,f(x))$ and vary $f$ by $f+\varepsilon\eta$. Then
\[
 \nu=\frac{(-f',1)}{\sqrt{1+f'^2}},\qquad
 h=\frac{\eta}{\sqrt{1+f'^2}},\qquad
 \dd s=\sqrt{1+f'^2}\,\dd x.
\]
In particular,
\begin{equation}\label{eq:vertical-to-normal}
h\,\dd s=\eta\,\dd x.
\end{equation}
Set
\begin{equation}\label{eq:w-def-linearization}
 w_f(x)=-\psi_{f,y}(x,f(x)).
\end{equation}
Differentiating $\psi_f(x,f(x))=0$ gives $\psi_{f,x}+f'\psi_{f,y}=0$, hence
\begin{equation}\label{eq:q-w-identity-linearization}
 -q_f=w_f\sqrt{1+f'^2}.
\end{equation}
The local and nonlocal parts of the quadratic form therefore satisfy
\begin{align}
 \int_\Gamma(-q)h^2\,\dd s
     &=\int_{-1}^1 w_f(x)\eta(x)^2\,\dd x,\label{eq:local-form-x}\\
 \iint_{\Gamma\times\Gamma}G_Hh(X)h(Y)\,\dd s_X\dd s_Y
     &=\iint_{(-1,1)^2}G_H(X(x),X(\xi))
                    \eta(x)\eta(\xi)\,\dd x\dd\xi.
                    \label{eq:nonlocal-form-x}
\end{align}
Thus no separate arc-length or slope factor remains in the graph quadratic
form:
\begin{equation}\label{eq:quadratic-vertical}
\mathcal Q_f[h]=\int_{-1}^1w_f\eta^2\,\dd x
  -\iint_{(-1,1)^2}G_H(X(x),X(\xi))\eta(x)\eta(\xi)\,\dd x\dd\xi.
\end{equation}
The identity used later for the spectral measure is
\begin{equation}\label{eq:intrinsic-invariant-measure}
\frac{\dd s}{-q_f}=\frac{\dd x}{w_f(x)}.
\end{equation}

\subsection{Contact geometry and the logarithmic height}
\label{sec:contact-theory}
Let $r=1-x$ be the horizontal distance to the right contact. The explicit
height band gives
\begin{equation}\label{eq:contact-shape-scale}
 c\,r\log(e/r)\le f(x)\le C\,r\log(e/r)
                \qquad(0<r<r_0)
\end{equation}
for fixed positive constants. To obtain this, write
$1-x^2=r(2-r)$ in \eqref{prof:eq:T-main}; the regularized heights of the two
barriers are bounded above and below by positive constants. This argument
is a height estimate only. It does not differentiate an arbitrary graph in
the order interval.

For an exact profile the coefficient $q_f$ is negative on the open free
arc and tends to zero at contact. In fact, the continuity of the potential
gradient proved above gives $\partial_x\phi_f(1,0)=0$, since
$\phi_f(x,0)=0$, and $\partial_y\phi_f(1,0)=W$. Thus
$\nabla\psi_f\to0$ at the contact. In particular no positive lower bound
for $-q_f$ holds on the closed free arc. We use the exact weighted measure
rather than assigning an unproved rate to this vanishing coefficient.

The two useful contact quantities are the Jacobi measure and logarithmic
height:
\begin{equation}\label{eq:contact-mass-coordinate}
 \dd\mu_f(x)=\frac{\dd x}{w_f(x)},\qquad t=-\log f(x).
\end{equation}
They are not the same coordinate. On the positive half-arc the exact profile
is strictly decreasing, so logarithmic height parametrizes that arc away
from its central point. Write $x=x_f(t)$ and
$X_f(t)=(x_f(t),e^{-t})$. The boundary equation implies
\[
 \psi_x=-f'\psi_y=f'w_f,\qquad
 \frac{\dd t}{\dd x}=-\frac{f'}f.
\]
Consequently the density of the Jacobi measure in logarithmic height is
\begin{equation}\label{eq:contact-density-identity}
 J_f(t):=\frac{\dd\mu_f}{\dd t}
     =\frac{f(x_f(t))}{-\psi_x(X_f(t))}
     =\frac1{-\Phi_x(f;X_f(t))}.
\end{equation}
The last expression uses $\Phi_x=\phi_x/f$ and $\psi_x=\phi_x$.
The prime in $f'$ is a physical horizontal derivative. Formula
\eqref{eq:contact-density-identity} follows from the exact boundary equation,
not from the two height inequalities alone.

\subsubsection{A positive integral for the contact density}
The positivity entering this formula can be made explicit. Let $a(\eta)>0$
be the positive horizontal endpoint of the source slice at height $\eta$.
Then
\[
 \Omega_f=\{(\xi,\eta):0<\eta<f(0),\ -a(\eta)<\xi<a(\eta)\}.
\]
Since $\partial_xG_H=-\partial_\xi G_H$, integration of the horizontal
source variable gives
\[
 -\phi_x(x,y)=\int_0^{f(0)}
 \{G_H((x,y),(a(\eta),\eta))-
                 G_H((x,y),(-a(\eta),\eta))\}\dd\eta.
\]
For $x,y>0$ the integrand equals
\begin{equation}\label{eq:contact-positive-log}
 \frac1{4\pi}\log\left(
 1+\frac{16xya(\eta)\eta}
 {((x-a(\eta))^2+(y-\eta)^2)
  ((x+a(\eta))^2+(y+\eta)^2)}\right)>0.
\end{equation}
The identity is obtained by multiplying the two Green ratios and subtracting
their cross-products. At a logarithmic zero it is interpreted by integration.
One may restrict the range of integration to obtain a lower bound, because
every retained integrand is nonnegative.

The quantitative estimate needed later is
\begin{equation}\label{eq:contact-density-input}
 -\Phi_x(f;X_f(t))\ge\frac{0.82}{\pi}t,
 \qquad t\ge9,
\end{equation}
uniformly for every exact profile in the barrier class. Consequently
\[
        0<J_f(t)\le\frac{\pi}{0.82\,t},\qquad t\ge9.
\]
For $9\le t\le100$ this lower bound is obtained by rigorous interval
enclosures of the positive integral in \eqref{eq:contact-positive-log}. The
remaining range $t\ge100$ is proved analytically in the next subsection, so
there is no numerical extrapolation toward the contact point.

\subsubsection{Continuation of the density estimate to arbitrary depth}
The density estimate is computed only on a finite interval of logarithmic
heights. The remaining range follows analytically, so no mesh accumulates at
contact.

Let $a(\eta)$ be the positive horizontal endpoint of the source slice of
height $\eta$.  For small integration heights the barrier geometry gives
\begin{equation}\label{eq:far-gap-linear}
 1-a(\eta)\le k\eta
\end{equation}
with a fixed $k<1$.  Put $b=e^{-9}$, let the target height be $y=e^{-t}$,
and retain in the positive integral \eqref{eq:contact-positive-log} only
integration heights $y\le\eta\le b$.  Since the inverse profile is monotone,
$x_f(t)=a(y)\ge a(\eta)$.  Hence
\[
 x_f(t),\ a(\eta)\ge 1-kb,\qquad
 |x_f(t)-a(\eta)|\le 1-a(\eta)\le k\eta,
 \qquad |y-\eta|\le\eta.
\]
The first denominator in \eqref{eq:contact-positive-log} is therefore at
most $(1+k^2)\eta^2$.  The reflected denominator is at most
$4(1+b^2)$, while the numerator is at least
$16(1-kb)^2y\eta$.  Consequently the fraction inside the logarithm is
bounded below by
\[
 C\frac y\eta,
 \qquad
 C=\frac{4(1-kb)^2}{(1+k^2)(1+b^2)}.
\]
This is the link between the inverse-profile geometry and the physical
density; no extrapolation from a finite table is involved.  We obtain
\[
 -\Phi_x(f;X_f(t))\ge \frac1{4\pi}\int_y^b
      \log\left(1+C\frac y\eta\right)\dd\eta.
\]
Using $\log(1+z)\ge z/(1+z)$ and integrating explicitly,
\[
 \int_y^b \frac{Cy}{\eta+Cy}\,\dd\eta
 =Cy\log\frac{b+Cy}{(1+C)y}.
\]
Since $b=e^{-9}$ and $y=e^{-t}$, this gives
\begin{equation}\label{eq:far-density-analytic}
 \frac{\pi[-\Phi_x(f;X_f(t))]}{t}
 \ge \frac C4\left(1-\frac{9+\log(1+C)}{t}\right),\qquad t\ge100.
\end{equation}
The right-hand side increases with $t$. On the range of values of $C$ used
here it is also increasing in $C$. The verified lower bound for the geometric
constant makes the right-hand side larger than $0.82$ at $t=100$, and hence
for every $t\ge100$. This completes the proof of
\eqref{eq:contact-density-input} at arbitrary depth.

\subsubsection{The kernel after the unitary change}
If $v\in L^2(\dd\mu_f)$, define
$\widehat v(t)=J_f(t)^{1/2}v(x_f(t))$ on a contact chart. Then
\[
 \int |v|^2\dd\mu_f=\int|\widehat v(t)|^2\dd t.
\]
The corresponding kernel in $L^2(\dd t)$ is
\[
 \widehat G_f(t,s)=\sqrt{J_f(t)J_f(s)}\,
                      G_H(X_f(t),X_f(s)).
\]
For the same contact, dropping the horizontal separation from the denominator
in the Green ratio gives
\[
 0\le G_H(X_f(t),X_f(s))
 \le\frac1{2\pi}\log\frac{e^{-t}+e^{-s}}{|e^{-t}-e^{-s}|}
 =\frac1{2\pi}\log\coth\frac{|t-s|}{2}.
\]
Thus \eqref{eq:contact-density-input} gives the explicit majorant
\begin{equation}\label{eq:model-contact-kernel-explained}
 0\le\widehat G_f(t,s)\le\frac1{c_*}K_0(t,s),\qquad
 K_0(t,s)=\frac1{2\sqrt{ts}}\log\coth\frac{|t-s|}{2}.
\end{equation}
The prefactor comes from the measure transformation. In particular,
logarithmic height is not substituted for accumulated Jacobi mass without
its density. The interaction between opposite contacts has positive
horizontal separation and is bounded separately below.

\begin{lemma}\label{lem:model-tail}
For $S>0$, the kernel $K_0$ is square integrable on $(S,\infty)^2$ and
\[
 \|K_0\|_{L^2(S,\infty)\to L^2(S,\infty)}\le\frac{\pi^2}{4S}.
\]
Moreover, with $\ell(r)=\log\coth(r/2)$,
\begin{equation}\label{eq:logcoth-moment}
 \int_0^\infty r\ell(r)^2\dd r<2.
\end{equation}
\end{lemma}
\begin{proof}
For $r>0$,
\[
 \ell(r)=2\sum_{j=0}^\infty\frac{e^{-(2j+1)r}}{2j+1},\qquad
 \int_0^\infty\ell(r)\dd r=\frac{\pi^2}{4}.
\]
All summands are positive, so termwise integration follows by monotone
convergence. The Schur weight $t^{-1/2}$ gives
\[
 \frac1{t^{-1/2}}\int_S^\infty K_0(t,s)s^{-1/2}\dd s
 =\frac12\int_S^\infty\frac{\ell(|t-s|)}s\dd s
 \le\frac{\pi^2}{4S}.
\]
The symmetric Schur estimate proves the norm bound. For square integrability,
use symmetry and $t=s+r$ to obtain
\[
 \iint_{(S,\infty)^2}K_0(t,s)^2\dd t\dd s
 \le\frac1{2S}\int_0^\infty\ell(r)^2\dd r<\infty.
\]
The last integral is finite since $\ell(r)=O(1+|\log r|)$ at zero and
$\ell(r)=O(e^{-r})$ at infinity.

For \eqref{eq:logcoth-moment}, square the positive series and integrate.
Grouping terms with the sum of the two odd indices equal to $2m$ yields
\[
 \int_0^\infty r\ell(r)^2\dd r
 =\sum_{m=1}^\infty\frac1{m^3}
                 \sum_{j=1}^m\frac1{2j-1}
 \le\sum_{m=1}^\infty\frac1{m^2}=\frac{\pi^2}{6}<2.
\]
This also justifies the double-integral bound used for the adjacent
core--tail interaction.
\end{proof}

\paragraph{The contact singularity in the weighted space.}
There are two apparent singularities near contact. The local weight vanishes,
but that vanishing is built into the norm. The Green function has a
logarithmic diagonal singularity, but after the logarithmic-height change of
variables it becomes $\log\coth(|t-s|/2)$ multiplied by the density factors
above. The logarithm is square integrable across the diagonal and the density
estimate gives sufficient decay as $t,s\to\infty$. The nonlocal part is
therefore compact in the weighted space.

This is the analytic reason for the realization used here. Contact is not
being treated as a regular endpoint; rather, the Hilbert space is chosen so
that the true degeneracy has finite energy and the Green interaction is a
compact perturbation of the identity.

\subsection{The self-adjoint realization}\label{sec:selfadjoint}
Set $a=(-q_f)^{1/2}$. Multiplication by $a$ is the isometry
\begin{equation}\label{eq:Hmov-def}
 U:\mathcal H_f\longrightarrow L^2(\Gamma_f),\qquad Uh=ah.
\end{equation}
We write $\mathcal H_{\rm mov}=\mathcal H_f$ when the profile is fixed.
The inverse multiplication need not be bounded on unweighted $L^2$.
Define the conjugated operator directly by
\begin{equation}\label{eq:K-kernel-direct}
 Ku(X)=\int_\Gamma\frac{G_H(X,Y)}{a(X)a(Y)}u(Y)\,\dd s_Y.
\end{equation}

\begin{proposition}\label{prop:compact-selfadjoint}
The kernel in \eqref{eq:K-kernel-direct} belongs to $L^2(\Gamma\times\Gamma)$.
Thus $K$ is compact and self-adjoint. The quadratic form of the geometric
linearization has the continuous realization
\begin{equation}\label{eq:form-conjugation}
\mathcal Q_f[h]=\langle(I-K)Uh,Uh\rangle,\qquad
 Lh=a(K-I)Uh
\end{equation}
for regular $h$, and its form kernel is $U^{-1}\Ker(I-K)$.
\end{proposition}
\begin{proof}
On compact interior subarcs, $a$ is bounded above and below and the Green
kernel has only a logarithmic singularity. Near each contact, the weighted
kernel is controlled after the unitary logarithmic change of variables by
\eqref{eq:model-contact-kernel-explained}; the reflected term is regular.
Lemma~\ref{lem:model-tail}, together with the separated core--contact
estimates, gives square integrability there. A finite collection of these
charts covers $\Gamma\times\Gamma$. Symmetry of $G_H$ then proves
self-adjointness, and the Hilbert--Schmidt property gives compactness.

For regular $h$, set $u=ah$. Since $q=-a^2$, the kernel formula gives
$Sh=aKu$, and therefore $Lh=a(K-I)u$.
The corresponding form is bounded in the norm $\|ah\|_2$ and extends to
$\mathcal H_f$. Its form kernel is exactly $U^{-1}\Ker(I-K)$.
\end{proof}

The displacement $e_1\cdot\nu$ has nonzero contact traces. We retain it in
$\mathcal H_f$; a translation gauge will be imposed only after its simplicity
has been proved.

\subsection{The form domain and its completion}
The weighted space can be described without specifying a boundary trace.
The coefficient $a=(-q_f)^{1/2}$ is positive at every point of the open free
arc. Functions are identified up to arclength-null sets, and the two contact
endpoints themselves are null sets. For every $u\in L^2(\Gamma)$, the
measurable function $h=u/a$ on the open arc belongs to $\mathcal H_f$ and
satisfies $Uh=u$. Thus the map in \eqref{eq:Hmov-def} is onto as well as
isometric. In particular, $\mathcal H_f$ is complete.

For compactly supported smooth displacements define the bilinear form
\[
 \mathcal Q_f[h,k]=\int_\Gamma(-q_f)hk\dd s
   -\iint_{\Gamma\times\Gamma}G_H(X,Y)h(X)k(Y)\dd s_X\dd s_Y.
\]
The operator bound for $K$ gives
\[
 |\mathcal Q_f[h,k]|\le(1+\|K\|)
                    \|h\|_{\mathcal H_f}\|k\|_{\mathcal H_f}.
\]
It therefore has a unique continuous extension to the full weighted space.
Density can be seen directly: approximate $u=ah$ in $L^2$ first by a cutoff
away from the contacts and then by smooth functions on compact interior
arcs. Dividing by $a$, which is smooth and bounded away from zero on each
such arc, gives an approximating sequence in the geometric variable.

The form kernel means the bilinear kernel,
\[
 \Ker_{\rm form}L_f
   =\{h\in\mathcal H_f:\mathcal Q_f[h,k]=0
                         \text{ for every }k\in\mathcal H_f\}.
\]
It is not the set defined by the single scalar condition
$\mathcal Q_f[h]=0$: an indefinite quadratic form can vanish on vectors
which do not belong to its kernel. Since $U$ is onto, the bilinear definition
is equivalent to $(I-K)Uh=0$. This proves the kernel identification without
assuming that the formal multiplication by $a^{-1}$ is bounded in unweighted
$L^2$.

There is also a useful direct realization in the vertical graph variable.
With $\eta=h\sqrt{1+f'^2}$ and $w_f=-\psi_y$, define
\[
 \dd\mu_f(x)=\frac{\dd x}{w_f(x)},\qquad
       v(x)=w_f(x)\eta(x).
\]
Then
\begin{align*}
 \|h\|_{\mathcal H_f}^2
     &=\int_{-1}^1v(x)^2\dd\mu_f(x),\\
 \mathcal Q_f[h]
     &=\int_{-1}^1v^2\dd\mu_f
       -\iint G_H(X(x),X(\xi))v(x)v(\xi)
                           \dd\mu_f(x)\dd\mu_f(\xi).
\end{align*}
Thus the weighted graph operator has the unmodified geometric Green kernel
with respect to the Jacobi measure $\dd\mu_f$. A later change of independent
variable changes both this measure and the kernel; the corresponding
square-root Jacobians are included explicitly below.

\subsection{Finite-energy symmetry fields}
The two geometric variations used in the proof belong to the form domain.
The translation displacement satisfies $|e_1\cdot\nu|\le1$, and the dilation
displacement satisfies $|X\cdot\nu|\le|X|$. The patch is bounded, and monotonicity gives
$\operatorname{length}(\Gamma)\le2+2f(0)$. The gradient of the relative
stream function is bounded. Hence
\[
 \int_\Gamma(-q_f)(e_1\cdot\nu)^2\dd s<\infty,
 \qquad
 \int_\Gamma(-q_f)(X\cdot\nu)^2\dd s<\infty.
\]
The contact endpoints may move under these variations. In particular,
horizontal translation must not be excluded by a zero-trace convention
before its simplicity has been proved.

The function $g_I=a^{-1}x_2$ used for the impulse constraint is also in
$L^2(\Gamma)$. Its squared norm is
\[
 \int_\Gamma\frac{x_2^2}{-q_f}\dd s
       =\int_{-1}^1f(x)^2\dd\mu_f(x).
\]
On a contact chart, $f=e^{-t}$ and the density bound proved quantitatively
below is $\dd\mu_f/\dd t\le\pi/(c_*t)$. Thus the contact contribution
is bounded by a constant times $\int_9^\infty e^{-2t}t^{-1}\dd t$;
the compact part is finite because $w_f$ stays positive there. Consequently
the first variation of impulse is a continuous functional on the form space,
by Cauchy--Schwarz.

\section{Translation, dilation and the kernel}\label{sec:symmetry}
The two Euclidean symmetries have different spectral meanings. Horizontal
translation preserves the equation and the travelling speed, so its
derivative is a genuine zero mode. Dilation changes the travelling speed; its
derivative therefore solves an inhomogeneous Jacobi equation and has strictly
negative quadratic form. Once the even sector is known to contain at most one
nonpositive direction, these two identities determine the kernel completely.

The odd sector requires no numerical spectral information. On the positive
half-arc the translation field has one sign, and the reflected Green kernel
has the positivity needed for a ground-state transform. The resulting
sum-of-squares identity makes translation the unique odd zero mode. The
computer-assisted spectral comparison is needed only for the even sector.

Let $\mathcal R_v(x_1,x_2)=(-x_1,x_2)$. Reflection acts on boundary functions
by $(\mathcal R_vh)(X)=h(\mathcal R_vX)$. The boundary, $q_f$ and $G_H$ are
invariant, so $K$ commutes with $\mathcal R_v$. Hence
\[
 L^2(\Gamma)=\mathcal H_e\oplus\mathcal H_o,\qquad K=K_e\oplus K_o.
\]
The analogous decomposition holds in the weighted form space, and
\begin{equation}\label{eq:kernel-splitting}
\Ker_{\rm form}L=\Ker_{\rm form}L_e\oplus\Ker_{\rm form}L_o.
\end{equation}
Translation is odd under this reflection; dilation is even.

\subsection{The odd ground-state identity}\label{sec:odd}
The translated family has the same travelling speed. Its normal velocity
$t=e_1\cdot\nu$ therefore satisfies
\begin{equation}\label{eq:translation-mode}
Lt=0,\qquad u_t=at,\qquad Ku_t=u_t.
\end{equation}
Let $\Gamma_+$ be the part of the free arc with $x_1>0$. We use $L^2(\Gamma_+)$
for the reduced odd operator in the next calculation; extension to the full
odd space multiplies both quadratic form and norm squared by two.

Reflecting the interaction from the opposite half, the odd part of the single-layer operator is represented by the reduced kernel
\begin{equation}\label{eq:odd-kernel-def}
k_o(X,Y)=G_H(X,Y)-G_H(X,\mathcal R_v Y),
\qquad X,Y\in\Gamma_+.
\end{equation}
The sign of the reduced kernel follows directly from the half-plane Green function and does not require numerical information about the profile.  Write
\[
X=(x,z),\qquad Y=(y,s),\qquad x,y,z,s>0,
\]
for points in the interior of the right half-arc.  Using
\[
G_H(X,Y)=\frac1{4\pi}\log
\frac{(x-y)^2+(z+s)^2}{(x-y)^2+(z-s)^2},
\]
and replacing $y$ by $-y$ for $\mathcal R_v Y$, we obtain
\begin{align}
4\pi k_o(X,Y)
&=\log\frac{\big((x-y)^2+(z+s)^2\big)\big((x+y)^2+(z-s)^2\big)}
{\big((x-y)^2+(z-s)^2\big)\big((x+y)^2+(z+s)^2\big)}.
\end{align}
The numerator of the difference between the two cross-products is
\begin{align}
&\big((x-y)^2+(z+s)^2\big)\big((x+y)^2+(z-s)^2\big)\\
&\quad-\big((x-y)^2+(z-s)^2\big)\big((x+y)^2+(z+s)^2\big)
=16xyzs>0.
\end{align}
Consequently,
\begin{equation}\label{eq:odd-kernel-positive}
k_o(X,Y)>0
\qquad\hbox{for all interior }X,Y\in\Gamma_+.
\end{equation}
Since $f$ is strictly decreasing and $\nu$ is outward,
\begin{equation}\label{eq:t-positive}
t=e_1\cdot\nu>0
\qquad\hbox{in the interior of }\Gamma_+.
\end{equation}

The possible zeros of $t$ at the endpoints do not require a pointwise bound on $h/t$.  The ground-state identity is first proved for perturbations supported in a compact subarc on which $t$ is bounded away from zero.  For a general form-domain perturbation $h$, write $u=ah$ and choose cutoffs $\chi_j$ which vanish in shrinking endpoint neighborhoods and equal one on every fixed compact interior subarc for all sufficiently large $j$.  Set $u_j=\chi_ju$ and $h_j=a^{-1}u_j$.  Since $u_j\to u$ in $L^2(\Gamma,\dd s)$ and $K$ is bounded,
\[
\langle(I-K)u_j,u_j\rangle\longrightarrow
\langle(I-K)u,u\rangle.
\]
Thus $h_j\to h$ in the exact quadratic-form topology supplied by the self-adjoint conjugation; no separate endpoint bound for $h/t$ or for $q$ is being assumed.  The compactly supported identity applies to $h_j$.  If the limiting quadratic form is zero, then for every compact interior subarc $J\Subset\Gamma_+$ the nonnegative ground-state integral over $J\times J$ tends to zero and, for large $j$, coincides with the corresponding integral for $h$.  Hence $h/t$ is constant almost everywhere on $J$.  Exhausting $\Gamma_+$ by overlapping compact subarcs shows that one constant works on the whole open arc.  The identity $h=ct$ then holds in the form domain and, for continuous kernel elements, extends to the endpoints.

Since $Lt=0$, the equation for the ground state reads
\begin{equation}\label{eq:ground-equation}
-q(X)t(X)=\int_{\Gamma_+}k_o(X,Y)t(Y)\,\dd s_Y.
\end{equation}
Now let $h$ be any smooth odd perturbation.  On $\Gamma_+$ write
\[
h=t\xi.
\]

\begin{proposition}[Ground-state identity]\label{prop:ground-state}
For an odd perturbation $h=t\xi$ whose restriction to $\Gamma_+$ has
compact interior support,
\begin{equation}\label{eq:ground-state}
-\langle L_oh,h\rangle
=
\frac12\iint_{\Gamma_+\times\Gamma_+}
k_o(X,Y)t(X)t(Y)
\bigl(\xi(X)-\xi(Y)\bigr)^2
\,\dd s_X\dd s_Y.
\end{equation}
The nonnegativity and equality characterization extend to the full odd
form space by the cutoff argument above. Equality is possible only when
$h$ is a multiple of $t$.
\end{proposition}

The identity has a simple interpretation. The translation field is a positive
ground state on the right half-arc. Dividing an odd perturbation by that
ground state removes the zeroth-order part of the energy; what remains
measures only differences of the quotient between two boundary points. Since
both the reduced kernel and the ground state are strictly positive in the
interior, zero energy forces that quotient to be constant. Thus the odd
kernel is determined by symmetry and positivity, not by a numerical spectral
calculation.

\begin{proof}
Using $h=t\xi$ in the odd quadratic form gives
\[
-\langle L_oh,h\rangle
=\int_{\Gamma_+}(-q)t^2\xi^2\,\dd s
-\iint_{\Gamma_+\times\Gamma_+}
k_o(X,Y)t(X)t(Y)\xi(X)\xi(Y)\,\dd s_X\dd s_Y.
\]
Multiply \eqref{eq:ground-equation} by $t(X)\xi(X)^2$ and integrate in $X$.  This rewrites the diagonal term as
\[
\int_{\Gamma_+}(-q)t^2\xi^2\,\dd s
=\iint k_o(X,Y)t(X)t(Y)\xi(X)^2\,\dd s_X\dd s_Y.
\]
By symmetry of $k_o$, the same expression is unchanged if $\xi(X)^2$ is replaced by $\xi(Y)^2$.  Averaging the two representations and subtracting the mixed term yields
\[
\frac12\iint k_o(X,Y)t(X)t(Y)
\left(\xi(X)^2+\xi(Y)^2-2\xi(X)\xi(Y)\right)
\,\dd s_X\dd s_Y,
\]
which is exactly \eqref{eq:ground-state}.

Because $k_o(X,Y)t(X)t(Y)>0$ for interior points, equality implies
\[
\xi(X)=\xi(Y)
\]
for almost every pair $(X,Y)$.  Connectivity of $\Gamma_+$ then gives $\xi\equiv c$.
\end{proof}

\begin{corollary}\label{cor:odd-kernel}
The odd kernel is one-dimensional:
\begin{equation}\label{eq:odd-kernel-one}
\Ker L_o=\operatorname{span}\{e_1\cdot\nu\},
\qquad
\Ker(K_o-I)=\operatorname{span}\{u_t\}.
\end{equation}
\end{corollary}

\begin{proof}
The translation mode belongs to the kernel, so the right-hand side is contained in the left-hand side.  Conversely, if $L_oh=0$, then the left-hand side of \eqref{eq:ground-state} vanishes.  Hence $h=t\xi$ with $\xi$ constant, and therefore $h$ is a multiple of $t$.
\end{proof}

\subsection{Dilation}\label{sec:even}
Translation gives a homogeneous Jacobi field because the speed is unchanged.
Dilation is different. If the patch is enlarged by a factor close to one, its
stream function scales quadratically while its travelling speed scales
linearly. Differentiating this family therefore leaves a nonzero right-hand
side in the Jacobi equation. Pairing that equation with the dilation
variation produces a sign which can be evaluated by the divergence theorem.
This is the source of the single negative even direction.

Consider the one-parameter family of dilated patches
\[
\Omega_\lambda=\lambda\Omega,
\qquad \lambda>0,
\]
with the vorticity strength kept equal to one.  If $\phi$ is the stream function of $\Omega$, then the stream function of $\Omega_\lambda$ is
\begin{equation}\label{eq:scaled-stream}
\phi_\lambda(x)=\lambda^2\phi(x/\lambda).
\end{equation}
Indeed, $-\Delta\phi_\lambda={\bf1}_{\lambda\Omega}$ and the half-plane boundary condition is preserved.  On a scaled boundary point $x=\lambda X$ we therefore have
\[
\phi_\lambda(\lambda X)=\lambda^2\phi(X)
=\lambda^2W X_2=(\lambda W)(\lambda X_2).
\]
Thus the travelling speed scales according to
\begin{equation}\label{eq:scaled-speed}
W_\lambda=\lambda W.
\end{equation}

Differentiate the pulled-back boundary equation
\[
\phi_\lambda(\lambda X)-W_\lambda(\lambda X_2)=0
\qquad\hbox{for }X\in\Gamma
\]
at $\lambda=1$.  The normal component of the boundary velocity is
\begin{equation}\label{eq:scale-mode}
z=X\cdot\nu,
\end{equation}
while $\dot W=W$.  By the full shape derivative formula \eqref{eq:full-linearization}, with no variation of the boundary level, the derivative is
\[
Lz-\dot W x_2=0.
\]
Hence
\begin{equation}\label{eq:scale-identity}
{Lz=Wx_2.}
\end{equation}
This derivation fixes both the coefficient and the sign.

Set
\[
u_z=az.
\]
Conjugating \eqref{eq:scale-identity} gives
\begin{equation}\label{eq:K-scale}
(K_e-I)u_z=a^{-1}Wx_2.
\end{equation}
No sign of the geometric scale displacement $z=X\cdot\nu$ is needed.  Indeed, pairing the scale identity with $z$ and using the divergence theorem gives
\begin{align}
\langle(I-K_e)u_z,u_z\rangle
&=-\langle Lz,z\rangle\notag\\
&=-W\int_\Gamma x_2(X\cdot\nu)\,\dd s\notag\\
&=-3W\int_\Omega x_2\,\dd x<0.\label{eq:scale-rayleigh-negative}
\end{align}
Here the boundary portion on $x_2=0$ contributes zero, and
$\operatorname{div}(x_2X)=3x_2$ in two dimensions.  Thus the Rayleigh quotient of $K_e$ at $u_z$ is strictly larger than one.

\begin{proposition}\label{prop:mu1-above}
Let $\mu_{1,e}=\sup\sigma(K_e)$ be the top eigenvalue of the compact self-adjoint operator $K_e$ on the even space.  Then
\begin{equation}\label{eq:mu1-strict}
\mu_{1,e}>1.
\end{equation}
\end{proposition}

\begin{proof}
Equation \eqref{eq:scale-rayleigh-negative} gives
\[
\frac{\langle K_eu_z,u_z\rangle}{\|u_z\|_2^2}>1.
\]
The variational characterization of the top spectral value therefore yields $\mu_{1,e}>1$.  
\end{proof}

\subsection{The even spectral reduction}
Let $\mu_{2,e}$ denote the second even spectral value of $K$.
\begin{proposition}\label{prop:spectral-reduction}
If $\mu_{2,e}<1$, then the even form has index one and trivial kernel.
Consequently
\begin{equation}\label{eq:full-kernel-final}
 \Ker_{\rm form}L=\operatorname{span}\{e_1\cdot\nu\}.
\end{equation}
\end{proposition}
\begin{proof}
The dilation identity gives $\mu_{1,e}>1$. If $\mu_{2,e}<1$, no even
spectral value equals one and exactly one exceeds it. Combine this with
Corollary~\ref{cor:odd-kernel}.
\end{proof}
The coercivity estimate of Proposition~\ref{prop:principal-coercivity}
provides the required separation below one. We now prove that estimate.

\section{Comparison on a compact core}\label{sec:operator-comparison}

We now prove the estimate which excludes a second nonpositive even direction
away from contact.  On the positive half of the free boundary we choose a
fixed point $x_c<1$.  The interval
\[
        0\le x\le x_c
\]
is the \emph{core}; by even reflection it represents the corresponding
compact central part of the full free boundary.  The two small pieces
$x_c<x<1$ near the contact points are not discarded.  They are treated in
the next section and then coupled back to the core.

The advantage of the core is that the endpoint degeneracy is absent there.
For an exact profile, $w_f=-\psi_{f,y}$ is positive on this compact interval,
and the Green kernel has only its ordinary logarithmic singularity on the
diagonal.  Our goal is to prove that, after imposing one linear condition on functions
supported in the core, the quotient in \eqref{eq:rayleigh-x-before-u} is
strictly smaller than $1$.

\subsection{Putting all profiles in one coordinate}

Let $f\in\mathcal K$ be an exact profile and write
\[
        X_f(x)=(x,f(x)),\qquad 0<x<1.
\]
For even perturbations the interaction on the positive half-arc is described
by
\begin{equation}\label{eq:even-half-kernel-physical}
 G_{e,f}^+(x,\xi)
 =G_H(X_f(x),X_f(\xi))+G_H(X_f(x),X_f(-\xi)).
\end{equation}
The second term is simply the interaction with the reflected point on the
left half of the free boundary.  In the vertical graph variable $\eta$ the
even Rayleigh quotient is
\begin{equation}\label{eq:rayleigh-x-before-u}
 \mathscr R_f(\eta)=
 \frac{\displaystyle\int_0^1\!\!\int_0^1
       G_{e,f}^+(x,\xi)\eta(x)\eta(\xi)\,\dd x\dd\xi}
      {\displaystyle\int_0^1 w_f(x)\eta(x)^2\,\dd x}.
\end{equation}

A direct comparison in the variable $x$ is inconvenient because the natural
weight $w_f$ depends on the unknown exact profile.  We therefore choose once
and for all a positive reference density $p(x)$ determined by the explicit
reference graph $\bar f$, and introduce
\begin{equation}\label{eq:c15-fixed-u}
 u(x)=\int_0^x\frac{\dd\xi}{p(\xi)},\qquad
 \varphi(u)=p(x(u))\eta(x(u)),\qquad
 r_f(u)=\frac{w_f(x(u))}{p(x(u))}.
\end{equation}
The particular formula for $p$ is not conceptually important; what matters is
that it is explicit, fixed independently of $f$, and positive on the core.
Since $\dd x=p\,\dd u$, the quotient becomes
\begin{equation}\label{eq:c15-rayleigh}
 \mathscr R_f(\eta)=
 \frac{\displaystyle\iint G_{e,f}^+(u,v)
       \varphi(u)\varphi(v)\,\dd u\dd v}
      {\displaystyle\int r_f(u)\varphi(u)^2\,\dd u},
\end{equation}
where $G_{e,f}^+(u,v)=G_{e,f}^+(x(u),x(v))$.  All exact profiles now live on
the same interval and with the same Lebesgue measure.  Their dependence is
confined to the kernel and to the positive multiplier $r_f$.

Choose $H>0$ so that $x_c=x(H)$ and denote by $A_f^{\rm core}$ the integral
operator with kernel $G_{e,f}^+$ on $L^2(0,H)$.  The precise value of the cut
is fixed once and for all; no optimization in $f$ is involved.

\subsection{The core transfer argument}

The logic of the core estimate is independent of the numerical values.  Let
$\mathsf M$ be a finite-rank self-adjoint operator on $L^2(0,H)$, obtained by
using normalized piecewise-constant functions on a fixed finite partition.
It is represented by a real symmetric matrix $M$ and is extended by zero on
the orthogonal complement of the cell space.  We use $\mathsf M$ only as a
comparison operator for the continuous reference profile $\bar f$.

Suppose that four bounds are available:
\begin{equation}\label{eq:core-abstract-inputs}
 \lambda_2(M)\le\lambda_*,\qquad
 \|A_{\bar f}^{\rm core}-\mathsf M\|\le\varepsilon_{\rm proj},\qquad
 \|A_f^{\rm core}-A_{\bar f}^{\rm core}\|\le\varepsilon_{\rm geom},
\end{equation}
and
\begin{equation}\label{eq:r0-abstract}
        r_f(u)\ge r_0>0\qquad(0\le u\le H).
\end{equation}
Let $v_1$ be a normalized top eigenvector of $M$ and set
\[
        X_0=\{\varphi\in L^2(0,H):\langle\varphi,v_1\rangle=0\}.
\]
This is the single condition imposed on the core.  Since $\mathsf M$ has at
most one eigenvalue larger than $\lambda_*$,
\[
        \langle\mathsf M\varphi,\varphi\rangle
        \le\lambda_*\|\varphi\|_2^2,
        \qquad \varphi\in X_0.
\]
The two operator-norm comparisons then give
\[
 \langle A_f^{\rm core}\varphi,\varphi\rangle
 \le(\lambda_*+\varepsilon_{\rm proj}+\varepsilon_{\rm geom})
       \|\varphi\|_2^2.
\]
Together with \eqref{eq:r0-abstract}, this yields
\begin{equation}\label{eq:core-transfer-rayleigh}
 \frac{\langle A_f^{\rm core}\varphi,\varphi\rangle}
      {\displaystyle\int_0^H r_f(u)\varphi(u)^2\,\dd u}
 \le
 \frac{\lambda_*+\varepsilon_{\rm proj}+\varepsilon_{\rm geom}}{r_0},
 \qquad \varphi\in X_0\setminus\{0\}.
\end{equation}
Thus the whole core argument reduces to proving that the number on the right
is strictly smaller than $1$.

\begin{proposition}[Quantitative core input]\label{prop:spectral-data}
For every exact profile $f\in\mathcal K$, the four bounds above hold with
\[
 \lambda_*<0.7684,\qquad
 \varepsilon_{\rm proj}<0.0343,\qquad
 \varepsilon_{\rm geom}<0.078,\qquad
 r_0>0.8903.
\]
Consequently
\begin{equation}\label{eq:core-gap-simple}
 \frac{\lambda_*+\varepsilon_{\rm proj}+\varepsilon_{\rm geom}}{r_0}
 <0.99<1.
\end{equation}
\end{proposition}

This proposition is the only quantitative input needed from the core.  Its
meaning is simple: after removing one direction, the exact core operator has
a uniform gap below the critical level $1$.

\begin{proposition}[Core spectral gap]\label{prop:core-transfer-detailed}
For every exact profile $f\in\mathcal K$, the truncated even operator on the
core has at most one spectral direction at or above level $1$.  Equivalently,
\eqref{eq:core-transfer-rayleigh} is $<1$ on the fixed codimension-one space
$X_0$.
\end{proposition}
\begin{proof}
Insert the bounds of Proposition~\ref{prop:spectral-data} into
\eqref{eq:core-transfer-rayleigh}.
\end{proof}

We next explain the analytic content behind the four estimates.  The exact
finite schedules and interval outputs are verification data; they are not
needed to understand the reduction.

\subsection{What is proved for the finite matrix}

We need to know that the symmetric matrix $M$ has at most one eigenvalue
larger than a number $\lambda_*<1$.  This is checked without relying on a
floating-point eigensolver.  We rigorously factor
\[
        M-\lambda_*I=LDL^T,
\]
with $L$ invertible and every diagonal entry of $D$ separated from zero.
The signs of the diagonal entries of $D$ then give the number of eigenvalues
of $M$ above $\lambda_*$.  The calculation gives exactly one.  Therefore,
after imposing orthogonality to the top eigenvector, one has
\[
        \langle\mathsf M\varphi,\varphi\rangle
        \le \lambda_*\|\varphi\|_2^2.
\]

\subsection{From the finite matrix to the continuous reference operator}

The Green kernel has a logarithmic diagonal singularity.  We do not bound
that singularity pointwise on a cell.  Instead we integrate its singular
part exactly.  If
\[
        F(s)=\frac{s^2}{2}\log|s|-\frac{3s^2}{4},\qquad F(0)=0,
\]
then $F''(s)=\log|s|$, and therefore rectangle integrals of
$\log|u-v|$ and of the reflected term $\log(u+v)$ are explicit.  After these
singular pieces are removed, the remainder is bounded directly as follows.
For a symmetric error kernel $R$, an estimate
\[
        \int |R(u,v)|\vartheta(v)\,\dd v
        \le E\vartheta(u)
\]
implies that the corresponding operator on $L^2(0,H)$ has norm at most $E$.
This is how the continuous reference operator is compared with the
particular finite-rank operator $\mathsf M$.

The important point is that no eigenfunction of the exact vortex is being
approximated.  The matrix gives the inequality on the finite-dimensional
space, and the displayed integral estimate transfers that inequality to the
continuous operator for $\bar f$.

\subsection{From the reference profile to an arbitrary exact profile}

The exact profile $f$ is not replaced by $\bar f$.  Instead we use the
barrier localization to estimate the difference of their Green kernels.
At fixed horizontal points the direct logarithmic terms differ by an
integrable quantity of the form
\[
        \log\!\left(1+\frac{D}{|x-\xi|^2}\right),
\]
with $D$ controlled by the pointwise height error.  This logarithmic
majorant is integrable across the diagonal.  Applying the same integral estimate to this difference gives
\[
 \|A_f^{\rm core}-A_{\bar f}^{\rm core}\|
 \le \varepsilon_{\rm geom}.
\]
No uniform slope estimate for every graph in the barrier class is required.

\subsection{The denominator is uniformly positive on the core}

The last issue is the denominator in \eqref{eq:c15-rayleigh}.  At an exact
boundary point, the identity $\psi_y=f\Phi_y$ gives
\begin{equation}\label{eq:rf-positive-identity}
 r_f(u)=\frac{f(x(u))[-\Phi_y(f;x(u),f(x(u)))]}{p(x(u))}.
\end{equation}
Thus it is enough to obtain a uniform positive lower bound for $-\Phi_y$ on
the core.

The useful positivity appears only after one integrates first along a whole
vertical column of the patch.  That complete-column contribution to
$-\Phi_y$ is nonnegative.  We may therefore discard columns and retain a
fixed set of columns lying well inside every admissible patch.  On those
retained columns monotonicity of $f$ places the integration height below the
target height, and the remaining differentiated kernel is pointwise
nonnegative.  Integrating over a fixed interior rectangle gives a uniform
positive lower bound.  Combining this with the explicit bounds for $f/p$
produces the last estimate $r_f\ge r_0$ in
Proposition~\ref{prop:spectral-data}.

This completes the compact-core argument.  Everything that remains concerns
the pieces of the free boundary approaching the two contact points.

\section{The contact tails and coercivity}\label{sec:core-tail-abstract}
The finite-dimensional comparison controls only a compact part of the
boundary. The touching endpoints cannot be discarded, because the form domain
allows finite-energy displacements there. We therefore split the exact
operator at a fixed physical point into a core, a tail, and their interaction.

Only two facts are needed about the tail. Its own operator norm must be below
one, so the isolated tail has positive energy, and the core--tail coupling
must be small enough to be absorbed by that positive energy. A two-by-two
Schur argument then joins the compact-core gap to the infinite contact tail.
We display all numerical inputs before they are used, so the reader can follow
the argument without consulting a later appendix.

\subsection{A common Hilbert-space realization}
For the exact profile define the Jacobi measure as in
\eqref{eq:contact-mass-coordinate}. If an accumulated mass coordinate is
introduced, its definition is
\begin{equation}\label{eq:intrinsic-mf}
 m_f(x)=\int_0^x\frac{\dd\xi}{w_f(\xi)}.
\end{equation}
No numerical value of $m_f$ is identified with logarithmic height. For the
reference coordinate $u$, one has exactly
\begin{equation}\label{eq:dm-du-rf}
 \dd\mu_f=\frac{\dd u}{r_f(u)}.
\end{equation}
On the core let $\varphi(u)=p(x(u))\eta(x(u))$ and set
\begin{equation}\label{eq:unitary-u-to-m}
 v(x(u))=r_f(u)\varphi(u)=w_f(x(u))\eta(x(u)).
\end{equation}
Then
\begin{align}
 \int v^2\dd\mu_f&=\int r_f\varphi^2\dd u,
                  \label{eq:unitary-denominator}\\
 \iint G_{e,f}^+v(x)v(\xi)\dd\mu_f(x)\dd\mu_f(\xi)
 &=\iint G_{e,f}^+(u,u')\varphi(u)\varphi(u')\dd u\dd u'.
                  \label{eq:unitary-numerator}
\end{align}
Thus the generalized reference-coordinate form is unitarily
identified with the Green operator on $L^2(\dd\mu_f)$. The image of $X_0$
is a closed codimension-one subspace of the exact core space and has the
same Rayleigh bound. Denote it by $X_{0,f}^{\mu}$.

\subsection{The physical cut and the exact block operators}
The reference partition fixes $x_c=x(H)$. Its logarithmic height for an
exact profile is
\begin{equation}\label{eq:Sc-def}
 S_c(f)=-\log f(x_c).
\end{equation}
The height enclosure at this fixed point, not the integral in
\eqref{eq:intrinsic-mf}, is used to prove the quantitative cut window.
In physical coordinates define
\begin{equation}\label{eq:He-mass-decomp}
 L^2((0,1),\dd\mu_f)=\mathcal H_c(f)\oplus\mathcal H_t(f),
\end{equation}
where
\begin{equation}\label{eq:Hc-Ht-def}
 \mathcal H_c(f)=L^2((0,x_c),\dd\mu_f),\qquad
 \mathcal H_t(f)=L^2((x_c,1),\dd\mu_f).
\end{equation}
Let $P_c$ and $P_t$ be multiplication by the characteristic functions of
these intervals. The even operator in this realization is
\begin{equation}\label{eq:Kem-kernel}
 (K_{e,f}^{\mu}v)(x)=\int_0^1G_{e,f}^+(x,\xi)v(\xi)\dd\mu_f(\xi).
\end{equation}
Set
\begin{equation}\label{eq:ABC-def}
 A_f=P_cK_{e,f}^{\mu}P_c,\qquad
 B_f=P_tK_{e,f}^{\mu}P_c,\qquad C_f=P_tK_{e,f}^{\mu}P_t.
\end{equation}
Then
\begin{equation}\label{eq:block-K-core-tail}
 K_{e,f}^{\mu}=\begin{pmatrix}A_f&B_f^*\\ B_f&C_f\end{pmatrix}.
\end{equation}
For example,
\begin{align}
 (A_fv)(x)&=\int_0^{x_c}G_{e,f}^+(x,\xi)v(\xi)\dd\mu_f(\xi),
                   &&x<x_c,\label{eq:A-integral}\\
 (B_fv)(x)&=\int_0^{x_c}G_{e,f}^+(x,\xi)v(\xi)\dd\mu_f(\xi),
                   &&x>x_c,\label{eq:B-integral}\\
 (C_fv)(x)&=\int_{x_c}^1G_{e,f}^+(x,\xi)v(\xi)\dd\mu_f(\xi),
                   &&x>x_c.\label{eq:C-integral}
\end{align}
The left contact is already included by even reflection in $G_{e,f}^+$.
There is no unrepresented second tail and no smooth cutoff error.
The core comparison gives
\begin{equation}\label{eq:A-core-codim-bound}
 \langle A_fv,v\rangle\le a_{\rm core}\|v\|^2
                      \qquad(v\in X_{0,f}^{\mu}).
\end{equation}

\subsection{Quantitative information at the core--tail cut}

The tail begins at the fixed physical point $x_c$. For every exact profile in
$\mathcal K$, its logarithmic height
\[
        S_c(f)=-\log f(x_c)
\]
satisfies
\begin{equation}\label{eq:tail-finite-inputs}
        15.99<S_c(f)<16.01.
\end{equation}
Together with \eqref{eq:contact-density-input}, the remaining uniform bounds
used below are
\begin{equation}\label{eq:tail-finite-inputs-2}
 \mu_f\{t\le0.53\}<308,\qquad
 \mu_f\{0.53\le t\le9\}<16.253,
\end{equation}
and
\begin{equation}\label{eq:tail-finite-inputs-3}
 f(0)<0.606,\qquad
 1-x_f(t)<6\times10^{-9}\quad(t\ge15.99).
\end{equation}
In particular the Jacobi mass up to logarithmic height $9$ is bounded by
\[
        M_9:=324.253,
\]
and in all formulas below we may take
\[
        c_*=0.82,\qquad S_-=15.99,\qquad S_+=16.01,
        \qquad Y_0=0.606,\qquad d_0=6\times10^{-9}.
\]
These are the only quantitative profile inputs needed for the tail and the
core--tail interaction. Once they are inserted, the estimates in the next two
subsections are purely analytic.

Integrating the density bound from height $9$ to the physical core edge gives
\begin{equation}\label{eq:core-mass-bound}
 \mu_f(0,x_c)\le M_9+\frac\pi{c_*}\log\frac{S_+}{9}.
\end{equation}

\subsection{The tail operator}
Split $C_f$ into same-contact and horizontally reflected interactions.
By \eqref{eq:model-contact-kernel-explained} and
Lemma~\ref{lem:model-tail}, the first has norm at most
\begin{equation}\label{eq:tail-same-bound}
 \|C_{\rm same}\|\le\frac{\pi^2}{4c_*S_-}.
\end{equation}
For the reflected interaction, the horizontal separation is at least
$\Delta=2-2d_0$. The inequality $\log(1+z)\le z$ gives
\[
 G_H((x,y),(-\xi,z))\le\frac{yz}{\pi\Delta^2}.
\]
This is a rank-one majorant on the tail. Consequently
\begin{align}
 \|C_{\rm reflected}\|
 &\le\frac1{\pi\Delta^2}\int_{S_c(f)}^\infty e^{-2t}J_f(t)\dd t\notag\\
 &\le\frac{e^{-2S_-}}{2c_*\Delta^2S_-}.
                    \label{eq:tail-reflected-bound}
\end{align}
In the last step $t^{-1}\le S_-^{-1}$ was used before integrating the
exponential. Adding \eqref{eq:tail-same-bound} and
\eqref{eq:tail-reflected-bound} bounds the full tail operator.

\subsection{The core--tail coupling}
We estimate three pieces of $B_f$. First consider points on the same side
with logarithmic heights $9\le t\le S_c(f)$ in the core and
$s\ge S_c(f)$ in the tail. After the unitary logarithmic change, their
kernel is bounded by $\ell(s-t)/(2c_*\sqrt{st})$. Its Hilbert--Schmidt norm
squared is therefore at most
\begin{align*}
 \frac1{4c_*^2\,9S_-}
 \int_9^{S_c(f)}\int_{S_c(f)}^\infty\ell(s-t)^2\dd s\dd t
 &\le\frac1{4c_*^2\,9S_-}\int_0^\infty r\ell(r)^2\dd r.
\end{align*}
For the second inequality use $a=S_c(f)-t$, $b=s-S_c(f)$ and enlarge both
nonnegative ranges. The measure of the pairs with $a+b=r$ is $r$.
It follows from \eqref{eq:logcoth-moment} that
\begin{equation}\label{eq:coupling-adjacent-bound}
 \|B_{\rm adjacent}\|\le
                  \left(\frac2{4c_*^2\,9S_-}\right)^{1/2}.
\end{equation}

For the same-side core with $t\le9$, the positive series of $\ell$ gives
\[
 G_H(X_f(t),X_f(s))\le
       \frac{e^{-(s-t)}}{\pi(1-e^{-2(s-t)})}
       \qquad(s\ge S_-,\ t\le9).
\]
Let $\epsilon=e^{-2(S_--9)}$. Bound the core mass by $M_9$, use the
contact density for $s$, and integrate the squared kernel. This gives
\begin{equation}\label{eq:coupling-far-bound}
 \|B_{\rm far}\|^2\le
      \frac{M_9\epsilon}{2\pi c_*S_-(1-\epsilon)^2}.
\end{equation}
The central point, at which the inverse-height chart degenerates, causes no
problem in this estimate: its neighbourhood is integrated in the measure
$\dd\mu_f$ and bounded by its total core mass.

Finally, the reflected tail has separation at least
$\Delta_c=1-d_0$ from every point of the positive core. Let $Y_0$ be a
uniform upper bound for the patch height and put
\[
 M_c=M_9+\frac\pi{c_*}\log\frac{S_+}{9}.
\]
The rank-one Green majorant and \eqref{eq:core-mass-bound} therefore imply
\begin{equation}\label{eq:coupling-opposite-bound}
 \|B_{\rm reflected}\|^2\le
       \frac{M_cY_0^2e^{-2S_-}}
                    {2\pi c_*\Delta_c^4S_-}.
\end{equation}
Each Hilbert--Schmidt bound is also an operator-norm bound. The triangle
inequality for these three pieces gives the total coupling bound. Splitting
the core in this way does not discard the interaction between the two
pieces; it writes the coupling operator as their sum and bounds that sum.

\begin{proposition}[Tail estimates]\label{prop:tail-estimates}
For every exact profile in $\mathcal K$, the exact blocks satisfy
\begin{equation}\label{eq:tail-abstract-bounds}
 \|C_f\|<0.188182,\qquad \|B_f\|<0.073711.
\end{equation}
\end{proposition}
\begin{proof}
Apply \eqref{eq:tail-same-bound}--\eqref{eq:tail-reflected-bound} and add
the square roots of
\eqref{eq:coupling-adjacent-bound}--\eqref{eq:coupling-opposite-bound}.
Substituting the numerical inputs displayed above into these analytic
expressions gives
\[
        \|C_f\|<0.188182,\qquad \|B_f\|<0.073711.
\]
At this point every dependence on the exact profile has already been removed;
the final check is only scalar arithmetic.
\end{proof}

The same estimates also supply the compactness required earlier. The
same-contact kernel is dominated by a square-integrable kernel on a deep
logarithmic chart. The separated opposite-contact and core--tail kernels
have the Hilbert--Schmidt bounds above. On a compact interior arc, the Green
kernel has an integrable squared logarithm and the weights are bounded.
These regions cover the product of the free arc with itself, and hence the
conjugated operator is Hilbert--Schmidt.

At this point the infinite-dimensional problem has been reduced to three
quantities: the core bound, the norm of the tail operator, and the norm of
the core--tail coupling. The following elementary block argument shows why
these are exactly the quantities that have to be controlled.

\subsection{A block-form estimate}
\begin{lemma}\label{lem:core-tail-threshold}
Let $K=\left(\begin{smallmatrix}A&B^*\\ B&C\end{smallmatrix}\right)$ be bounded
and self-adjoint on $\mathcal H_c\oplus\mathcal H_t$. Suppose that
$X_0\subset\mathcal H_c$ has codimension one and
\[
 \langle Ax,x\rangle\le\alpha\|x\|^2\quad(x\in X_0),\qquad
 \|B\|\le\beta,\qquad \|C\|\le\gamma<1.
\]
If
\begin{equation}\label{eq:abstract-final-condition}
 \alpha+\frac{\beta^2}{1-\gamma}<1,
\end{equation}
then $I-K$ is strictly positive on $X_0\oplus\mathcal H_t$.
More precisely, if $\delta>0$ satisfies
\begin{equation}\label{eq:block-delta}
 1-\alpha-\delta>0,\quad 1-\gamma-\delta>0,\quad
 (1-\alpha-\delta)(1-\gamma-\delta)>\beta^2,
\end{equation}
then
\begin{equation}\label{eq:block-coercivity}
 \langle(I-K)(x,y),(x,y)\rangle
       \ge\delta(\|x\|^2+\|y\|^2),\quad x\in X_0,\ y\in\mathcal H_t.
\end{equation}
\end{lemma}
\begin{proof}
Set $D=I-C$. Then $D\ge(1-\gamma)I$ and
$\|D^{-1}\|\le(1-\gamma)^{-1}$. Completing the square gives
\begin{align}
 \langle(I-K)(x,y),(x,y)\rangle
 ={}&\langle(I-A-B^*D^{-1}B)x,x\rangle\notag\\
 &+\|D^{1/2}(y-D^{-1}Bx)\|^2.\label{eq:schur-completion}
\end{align}
For $x\in X_0$, the first term is at least
\[
 \left(1-\alpha-\frac{\beta^2}{1-\gamma}\right)\|x\|^2.
\]
This is positive if $x\ne0$. If $x=0$ and $y\ne0$, positivity follows
from $D>0$. This proves the first assertion.

For the quantitative estimate, Cauchy--Schwarz gives
\begin{align*}
 &\langle(I-K)(x,y),(x,y)\rangle-\delta(\|x\|^2+\|y\|^2)\\
 &\qquad\ge(1-\alpha-\delta)\|x\|^2-2\beta\|x\|\|y\|
                        +(1-\gamma-\delta)\|y\|^2.
\end{align*}
The scalar quadratic form on the right is positive definite by
\eqref{eq:block-delta}. This gives \eqref{eq:block-coercivity}.
\end{proof}

\subsection{Proof of the even coercivity estimate}
\begin{proof}[Proof of Proposition~\ref{prop:principal-coercivity}]
The core estimate \eqref{eq:core-transfer-rayleigh} gives, on the transported
codimension-one core subspace,
\[
        \alpha:=\frac{0.7684+0.0343+0.078}{0.8903}.
\]
The tail estimates give
\[
        \beta:=0.073711,\qquad \gamma:=0.188182.
\]
Thus the block lemma applies provided
\[
        \alpha+\frac{\beta^2}{1-\gamma}<1.
\]
With the displayed decimal numbers interpreted as exact rationals,
\begin{equation}\label{eq:c35-exact-schur}
 \alpha+\frac{\beta^2}{1-\gamma}
 =\frac{7198053898471463}{7227615654000000}<1.
\end{equation}
Hence $I-K_e$ is strictly positive on a codimension-one subspace of the full
even form space.

For the quantitative coercivity constant set $\delta=1/250$. The two diagonal
terms $1-\alpha-\delta$ and $1-\gamma-\delta$ are positive, and
\begin{equation}\label{eq:c36-coercivity-determinant}
 (1-\alpha-\delta)(1-\gamma-\delta)-\beta^2
 =\frac{409740912537}{8903000000000000}>0.
\end{equation}
Therefore the $2\times2$ quadratic form in Lemma~\ref{lem:core-tail-threshold}
is positive definite with lower bound $\delta$. Pulling the corresponding
subspace back through the unitary coordinate changes and through
$h\mapsto(-q_f)^{1/2}h$ gives a closed codimension-one space
$X_f\subset\mathcal H_{f,e}$ on which
\[
        \mathcal Q_f[h]\ge\frac1{250}\|h\|_{\mathcal H_f}^2.
\]
This is \eqref{eq:principal-coercivity}.
\end{proof}

In particular, the nonpositive spectral subspace of $I-K_e$ has dimension
at most one. The dilation field has strictly negative form, so this
subspace consists of one negative eigendirection and contains no kernel.
Equivalently,
\begin{equation}\label{eq:even-gap-v62}
\mu_{2,e}<1.
\end{equation}
The min--max principle applied to the norm-coercivity estimate also gives
$\mu_{2,e}\le249/250$. Inequality \eqref{eq:c35-exact-schur} is the
Schur threshold; the stronger two-by-two positivity condition
\eqref{eq:c36-coercivity-determinant} yields the norm-coercivity constant
$1/250$.

\section{Nondegeneracy and the linear equation}\label{sec:proof-main-v40}
At this point all difficult estimates are complete. The odd sector has a
one-dimensional kernel generated by translation. The even sector has at most
one nonpositive direction, while dilation supplies one strictly negative
direction. Hence the even kernel is trivial. The remaining discussion records
the consequences for the constrained form and for solvability of the linear
equation.

\begin{proof}[Proof of Theorem~\ref{thm:main}]
Proposition~\ref{prof:thm:main} gives an exact profile in $\mathcal K$.
For any exact profile in this band, Corollary~\ref{cor:odd-kernel} identifies
the odd kernel with horizontal translation and proves nonnegativity of the
odd form. Dilation supplies a negative even direction by
\eqref{eq:scale-rayleigh-negative}. Proposition~\ref{prop:principal-coercivity}
shows that the full even nonpositive spectral space has dimension at most
one. It therefore consists of that negative eigendirection, and the even
kernel is trivial. The even--odd decomposition proves \eqref{eq:main-kernel}.
\end{proof}

\subsection{The fixed-impulse form}
The first variation of the upper-patch impulse is
\begin{equation}\label{eq:impulse-functional}
I(\Omega)=\int_\Omega x_2\,\dd X,
 \qquad DI(\Omega)[h]=\int_\Gamma x_2h\,\dd s.
\end{equation}
With $u=ah$, $a=(-q)^{1/2}$, put
\begin{equation}\label{eq:A-def}
A=I-K_e,\qquad g_I=a^{-1}x_2.
\end{equation}
The constraint is $u\perp g_I$. Theorem~\ref{thm:main} says that $A$ is
invertible with index one. For $z=X\cdot\nu$ and $u_z=az$, the dilation
identity gives
\begin{equation}\label{eq:Auw-gI}
 Au_z=-Wg_I.
\end{equation}
Hence
\begin{equation}\label{eq:exact-schur-impulse}
 \langle A^{-1}g_I,g_I\rangle
 =-\frac1W\int_\Gamma x_2(X\cdot\nu)\,\dd s
 =-\frac{3I(\Omega)}W<0.
\end{equation}

\begin{lemma}\label{lem:index-one-schur}
Let $A$ be bounded, self-adjoint and invertible on a real Hilbert space,
with $\ind A=1$. If $g\ne0$ and $\langle A^{-1}g,g\rangle<0$, then
$\langle Au,u\rangle>0$ for every nonzero $u\perp g$.
\end{lemma}
\begin{proof}
Set $y=A^{-1}g$. The vector $y$ has negative form, and $g^\perp$ is its
$A$-orthogonal complement. A negative direction in $g^\perp$ would therefore
give a two-dimensional negative subspace. Thus the restriction to $g^\perp$
is nonnegative. If $u\perp g$ has zero form, polarization of that nonnegative
restriction gives $\langle Au,v\rangle=0$ for all $v\perp g$, so $Au=cg$.
Also $\langle Au,y\rangle=\langle u,Ay\rangle=\langle u,g\rangle=0$.
Since $\langle g,y\rangle<0$, this implies $c=0$. Invertibility then gives
$u=0$.
\end{proof}

\begin{corollary}\label{prop:fixed-impulse-positivity}
For every nonzero even normal displacement in the form space satisfying
$\int_\Gamma x_2h\,\dd s=0$, one has $\mathcal Q_f[h]>0$.
\end{corollary}
\begin{proof}
Apply Lemma~\ref{lem:index-one-schur} with $g=g_I$ and use
\eqref{eq:exact-schur-impulse}.
\end{proof}

\subsection{Coercivity on the fixed-impulse space}
For a fixed exact profile the last corollary yields a norm-coercive estimate
on the even constraint, although its constant is not the explicit constant
$1/250$ on the comparison subspace. More precisely, there is a number
$c_f>0$ such that
\[
 \mathcal Q_f[h]\ge c_f\|h\|_{\mathcal H_f}^2
 \quad\text{if }h\text{ is even and }\int_\Gamma x_2h\dd s=0.
\]
To prove this, work in $g_I^\perp$ with $A=I-K_e$. If no such positive
constant exists, choose unit vectors $u_n$ in this subspace with
$\langle Au_n,u_n\rangle\to0$. After taking a subsequence, $u_n\rightharpoonup u$.
Compactness gives $K_eu_n\to K_eu$ strongly, and hence
$\langle K_eu,u\rangle=1$. In particular $u\ne0$. Weak lower
semicontinuity gives $\|u\|\le1$, so
$\langle Au,u\rangle\le0$. This contradicts strict positivity on
$g_I^\perp$.

The argument uses compactness after the exact constraint has been imposed.
The comparison subspace $X_f$ and the fixed-impulse tangent space need not
coincide. Therefore neither the numerical constant $1/250$ nor a numerical
bound for $c_f$ is transferred between them without a separate angle or
constraint estimate.

\subsection{Weighted solvability}
For a prescribed source $g$, the equation $Lh=g$ becomes
\begin{equation}\label{eq:weighted-inverse-eqn-v44}
 (K-I)u=a^{-1}g,\qquad u=ah.
\end{equation}
Since $\Ker(K-I)=\operatorname{span}\{a(e_1\cdot\nu)\}$, its compatibility
condition is
\begin{equation}\label{eq:compatibility-v44}
 \int_\Gamma g(e_1\cdot\nu)\,\dd s=0.
\end{equation}
If $a^{-1}g\in L^2$ and \eqref{eq:compatibility-v44} holds, there is a
unique solution with $u\perp a(e_1\cdot\nu)$. Writing $\delta_f>0$ for the
distance from one to the remaining spectrum of $K$, the spectral theorem
gives
\begin{equation}\label{eq:weighted-inverse-estimate-v44}
 \|(-q_f)^{1/2}h\|_{L^2(\Gamma_f)}
 \le\delta_f^{-1}\|(-q_f)^{-1/2}g\|_{L^2(\Gamma_f)}.
\end{equation}
For even $g$, compatibility is automatic. On compact subarcs away from
contact, $-q_f$ is bounded below and the logarithmic single-layer operator
has its usual local smoothing property, giving interior H\"older regularity.
The estimate is global in the weighted Hilbert space; it does not assert an
unweighted endpoint H\"older inverse or a nonlinear persistence theorem.

\subsection{Compatibility and the inverse on the complement}
The solvability condition can also be read in the dual form norm. Set
$F=a^{-1}g$ and $u_t=a(e_1\cdot\nu)$. The identity
\[
 \langle F,u_t\rangle=\int_\Gamma g(e_1\cdot\nu)\dd s
\]
is meaningful by Cauchy--Schwarz when $F\in L^2$. On $u_t^\perp$ the
self-adjoint operator $K-I$ is invertible. Indeed, if unit vectors $u_n$
in this complement satisfied $(K-I)u_n\to0$, compactness of $K$ would give
a subsequence for which $Ku_n$ converges strongly. The same would then be
true of $u_n$. Its nonzero limit would lie in $u_t^\perp\cap\Ker(K-I)$,
a contradiction. This proves a positive inverse bound on the complement
without assuming a positive lower bound for $-q_f$ at contact.

The solution is unique in the gauge $\langle ah,u_t\rangle=0$; other solutions
differ by a translation field. On an interior subarc the coefficient
$-q_f$ is bounded below, so the equation and local mapping properties of the
single-layer potential yield the usual interior regularity. At a contact
point the inverse estimate remains a weighted Hilbert estimate. A nonlinear
mapping theorem in an endpoint H\"older space is not a consequence of this
Hilbert statement alone.

\section{Computer-assisted verification}\label{prof:app:table}
We finish by describing the verification at the level needed to make the
argument reproducible.  The computational supplement contains the complete
sources, exact input data, interval outputs and scripts used for the finite inequalities stated in the main text.

All terminating decimals occurring in the rational profile and in the
thresholds are interpreted as exact rational numbers.  They are enclosed at
input and every subsequent arithmetic operation is evaluated outwards.
Elementary functions are evaluated with directed MPFR rounding
\cite{MPFR}; the interval calculation follows the standard inclusion
principle \cite{Interval}, and the basic operations satisfy the same contract.  Calculations which encounter a nonfinite value, a forbidden
division, a failed positivity condition or an incomplete source/target cover
are rejected rather than silently discarded.  These rules are part of the
certificate, not estimates added after a floating-point computation.

The profile verification follows exactly the decomposition of
Section~\ref{sec:profile-estimates}.  Near contact the analytic inequalities
are evaluated uniformly.  The transition region is treated on whole target
boxes.  On the regular ranges, the computer provides lower nodal integrals,
upper curvature bounds and source-extension errors, and the final cell
inequality is assembled with the full $2E$ loss.  The symmetry endpoint is
checked in the physical coordinate used in the proof.  Every source and
target interval is part of an exact finite cover.

For the spectral estimate, the stored matrix is regarded as exact finite
data.  Its inertia is certified by interval $LDL^T$ factorization.  The
projection and geometric errors are bounded by complete weighted row sums of
the continuous kernels after the singular logarithmic pieces have been
integrated analytically.  The denominator, contact density and mass bounds
are evaluated from the positive representations proved in the main text.
The final tail and coupling estimates then reduce to scalar inequalities,
which are also reconstructed with exact rational arithmetic independently of
the interval integration.

The ancillary archive provides a single verification entry point and records
the software environment used for the computations.  It also contains the
complete finite schedules and the exact identities of the proof-bearing
sources and outputs.  The purpose of these files is to permit reproduction
and independent checking of the finite enclosures.  No numerical
approximation, convergence plot or empirical error estimate is used as an
unproved hypothesis in the mathematical argument.

\appendix

\section{The rational profile and elementary kernel identities}\label{app:profile-details}

\subsection{Coordinates and polynomial data}

For $0\le x<1$ set
\begin{equation}\label{prof:eq:T}
 T=\tau(x)=\left(1+\log\frac1{1-x^2}\right)^{-1}.
\end{equation}
Thus $T=1$ at the symmetry axis $x=0$, while $T\downarrow0$ as $x\uparrow1$.  Put
\begin{equation}\label{prof:eq:omega}
 \omega_0(x)=(1-x^2)\left(1+\log\frac1{1-x^2}\right).
\end{equation}
In the $T$ variable,
\begin{equation}\label{prof:eq:inverse}
 1-x^2=e^{1-1/T},\qquad
 \omega_0(T)=\frac{e^{1-1/T}}{T},
\end{equation}
where $\omega_0(0)=0$ by continuity.  The distance from the right contact point is
\begin{equation}\label{prof:eq:dT}
 d(T)=1-x(T)
 =1-\sqrt{1-e^{1-1/T}}
 =\frac{e^{1-1/T}}{1+\sqrt{1-e^{1-1/T}}}.
\end{equation}
In particular,
\begin{equation}\label{prof:eq:contact-asymp}
 d(T)\sim\frac12e^{1-1/T},\qquad
 \omega_0(T)\sim\frac{2d(T)}{T}
 \quad (T\downarrow0).
\end{equation}
This is the natural scale for the logarithmic contact law.

We write
\begin{equation}\label{prof:eq:fg}
 f(x)=\omega_0(x)g(\tau(x)).
\end{equation}
A direct differentiation yields the exact normalized-slope identity
\begin{equation}\label{prof:eq:slope-id}
 -\frac{f'(x)}{x\ell(x)}
 =2\bigl((1-T)g(T)+T^2g'(T)\bigr),
 \qquad
 \ell(x)=1+\log\frac1{1-x^2}.
\end{equation}

Define
\begin{align}
 P_5(T)={}&0.62423019467308205
 -3.3516841475174495T
 +29.701069408864079T^2\notag\\
 &+37.854192680034771T^3
 +478.08782165674165T^4
 +857.31455373012113T^5,
 \label{prof:eq:P5}\\
 Q_5(T)={}&1
 -1.1208396736927408T
 +27.072844036240788T^2\notag\\
 &+139.09023100980983T^3
 +2209.0571013017056T^4
 -37.902053175375841T^5.
 \label{prof:eq:Q5}
\end{align}
The centre is
\begin{equation}\label{prof:eq:center}
 \bar g(T)=\frac{P_5(T)}{Q_5(T)},\qquad
 \bar f(x)=\omega_0(x)\bar g(\tau(x)).
\end{equation}
The coefficients in \eqref{prof:eq:P5}--\eqref{prof:eq:Q5} are henceforth exact decimals, that is, exact rational numbers with denominators powers of ten.

\subsection{The fixed reference density}
The positive reference density used in Section~\ref{sec:operator-comparison}
is $p(x)=\bar f(x)\gamma(T(x))$, where
\begin{align*}
 \gamma(T)={}&0.25421914265135914-0.041486845064598543T
                         +7.1706430618996144T^2\\
 &-36.128200316412432T^3+83.664615908816444T^4\\
 &-102.86499455869857T^5+64.955041223146324T^6
                         -16.581069755030139T^7.
\end{align*}
These coefficients are exact decimals. This polynomial is a coordinate
weight, not a replacement for the profile centre $P_5/Q_5$. Its positivity
on $[0,1]$ can be verified by exact Bernstein subdivision. On the retained
core both $\bar f$ and $\gamma$ are positive, so $u(x)=\int_0^x p^{-1}$
is strictly increasing and defines the common coordinate used for the
matrix and the continuous comparisons.

\subsection{Polynomial bounds}

We recall the elementary certificate used throughout.  If a polynomial $p$ is written on an interval $I=[a,b]$ in Bernstein form
\[
 p(a+(b-a)s)=\sum_{j=0}^n\beta_j\binom nj s^j(1-s)^{n-j},
 \qquad 0\le s\le1,
\]
then
\[
 \min_j\beta_j\le p(x)\le\max_j\beta_j
 \qquad (x\in I).
\]
All transformations to Bernstein form are performed in exact rational arithmetic.

\begin{proposition}[Sufficient exact bounds for the rational centre]\label{prof:prop:center}
On $[0,1]$,
\begin{equation}\label{prof:eq:Qbounds}
 Q_5>0,\qquad 0.05<\bar g<0.65.
\end{equation}
The reference graph is positive in $(-1,1)$ and strictly decreasing on
$(0,1)$.
\end{proposition}
\begin{proof}
Apply the Bernstein convex-hull property to $Q_5$ and to the rational
numerators for the stated height bounds.  The denominator positivity
allows multiplication by $Q_5$ without changing signs.  Monotonicity follows
from the two stronger barrier-slope checks in
Proposition~\ref{prof:prop:barrier-geometry}, since the centre is their
average.  Equivalently, its normalized slope has numerator
\begin{equation}\label{prof:eq:Sbar}
 S_0(T)=2\{(1-T)P_5Q_5+T^2(P_5'Q_5-P_5Q_5')\}.
\end{equation}
All conversions and subdivisions here have rational endpoints and exact
rational coefficients.  The polynomial checker provides the positive
Bernstein coefficients for the finite subinterval cover, not merely values
of the polynomials at endpoints.
\end{proof}

\subsection{The explicit barrier band}

Let $\rho$ be the continuous piecewise-affine function with the following
exact decimal data.  The two pairs of columns form one ordered list of knots.
\begin{center}
\begin{tabular}{@{}rr@{\qquad\qquad}rr@{}}
\toprule
$T$&$\rho(T)$&$T$&$\rho(T)$\\
\midrule
0&0.5&0.20&0.0022\\
0.03&0.5&0.30&0.0028\\
0.035&0.01&0.40&0.0032\\
0.04&0.005&0.50&0.0036\\
0.05&0.0025&0.60&0.004\\
0.075&0.0027&0.80&0.005\\
0.10&0.0016&1&0.006\\
0.15&0.0022&&\\
\bottomrule
\end{tabular}
\end{center}
Define
\begin{equation}\label{prof:eq:barriers}
 g_\pm(T)=\bar g(T)\pm\rho(T),\qquad
 f_\pm(x)=\omega_0(x)g_\pm(\tau(x)).
\end{equation}

\begin{proposition}[Exact geometry of the barriers]\label{prof:prop:barrier-geometry}
Both heights satisfy $g_\pm(T)>0.04$ on $[0,1]$.  On every affine panel of
$\rho$, with one-sided derivatives at the knots,
\begin{align}
 2\{(1-T)g_-(T)+T^2g_-'(T)\}&>0.04,\label{prof:eq:minusmono}\\
 2\{(1-T)g_+(T)+T^2g_+'(T)\}&>0.04.\label{prof:eq:plusmono}
\end{align}
Hence the physical barriers are positive in the interior, continuous at
all knots, and strictly decreasing on $(0,1)$.
\end{proposition}
\begin{proof}
On a panel write $\rho=\alpha+\beta T$.  If
$N_\sigma=P_5+\sigma(\alpha+\beta T)Q_5$, the two polynomial tests are
$N_\sigma-0.04Q_5>0$ and
\[
 2\{(1-T)N_\sigma Q_5+T^2(N_\sigma'Q_5-N_\sigma Q_5')\}
                          -0.04Q_5^2>0.
\]
Exact Bernstein conversion on a finite rational subdivision proves each
of these inequalities for both signs on all fourteen pieces, after
$Q_5>0$ has been checked separately.  There are 56 signed polynomial
claims.  The normalized physical-slope identity
\eqref{prof:eq:slope-id} then gives the assertion.  
\end{proof}

\begin{lemma}[Physical width]\label{prof:lem:width}
\begin{equation}\label{prof:eq:width}
 \sup_{0\le T\le1}\omega_0(T)\rho(T)=0.006.
\end{equation}
\end{lemma}

\begin{proof}
The function $\omega_0(T)=e^{1-1/T}/T$ is increasing on $(0,1]$.  For $T\ge0.1$, $\rho(T)\le0.006$, and thus $\omega_0\rho\le0.006$.  For $T\le0.1$, $\rho\le0.5$ while $\omega_0(T)\le\omega_0(0.1)<1.235\times10^{-3}$, giving $\omega_0\rho<6.18\times10^{-4}$.  Equality occurs at $T=1$.
\end{proof}

\subsection{Symmetrized formulas}\label{prof:app:kernels}

For the primitive \eqref{prof:eq:V}, put $A=a^2$, $z=yr$, and
\begin{equation}\label{prof:eq:Dsym}
 D=(A+b^2+z^2)^2-4b^2z^2.
\end{equation}
Then
\begin{equation}\label{prof:eq:Vint}
 V(a,y,b)=\frac1{4\pi}\int_0^1
 \log\frac{D}{(A+z^2)^2}\dd r.
\end{equation}
Differentiation under the integral gives
\begin{align}
 V_a(a,y,b)&=-\frac{ab^2}{\pi}\int_0^1
 \frac{A+b^2-3z^2}{D(A+z^2)}\dd r,\label{prof:eq:Va}\\
 V_y(a,y,b)&=-\frac1\pi\int_0^1
 \frac{rzb^2(3A+b^2-z^2)}{D(A+z^2)}\dd r.\label{prof:eq:Vy}
\end{align}
These identities also give
\begin{equation}\label{prof:eq:scale}
 V(a,y,b)=V(a/y,1,b/y).
\end{equation}

For $p>0$ and $a=p-d$, set
\[
 A=\frac{a}{y},\qquad P=\frac{p}{y},\qquad B=\frac{b}{y}.
\]
Then
\begin{equation}\label{prof:eq:dimensionless-VC}
{
 V(a,y,b)-C(p,b)
 =\frac1{4\pi}\int_0^1
 \log\!\left[
 \frac{P^4\bigl((A^2+B^2+r^2)^2-4B^2r^2\bigr)}
 {(P^2+B^2)^2(A^2+r^2)^2}
 \right] \dd r.}
\end{equation}
Equivalently, away from $a=0$ one may split this as
\begin{align}\label{prof:eq:dimensionless-split}
V(a,y,b)-C(p,b)
={}&\frac1{4\pi}\int_0^1
\left[
\log\!\left(1+\frac{2(A^2-B^2)r^2+r^4}{(A^2+B^2)^2}\right)
-2\log\!\left(1+\frac{r^2}{A^2}\right)
\right]\dd r \\
&+\frac1{2\pi}\log\!\left(
1+\frac{B^2D(2P-D)}{A^2(P^2+B^2)}
\right),\qquad D=\frac d y.\notag
\end{align}
Both identities follow by factoring the polynomial $D$ in
\eqref{prof:eq:Dsym} and using
$C(a,b)-C(p,b)=(2\pi)^{-1}\log\bigl((1+b^2/a^2)/(1+b^2/p^2)\bigr)$.

The first identity follows by integrating $J_a'$ over the two vertical source
segments. Its denominator factors as
\[
 D=(A+(b+z)^2)(A+(b-z)^2).
\]
Differentiating the logarithm gives \eqref{prof:eq:Va} and \eqref{prof:eq:Vy}.
The source-height derivative is
\begin{equation}\label{prof:eq:Vb-app}
 V_b(a,y,b)=\frac1{4\pi y}
        \log\frac{a^2+(y+b)^2}{a^2+(y-b)^2}.
\end{equation}
Subtracting $C_b(p,b)=b/(\pi(p^2+b^2))$ for the two reflected branches
yields \eqref{prof:eq:Fb}.

\subsection{The slope polynomials}\label{prof:app:bernstein}

On a barrier panel $[a,b]$, write $\rho(T)=\alpha+\beta T$.  Let
\[
 W=P_5'Q_5-P_5Q_5'.
\]
The centre-slope numerator is
\[
 S_0=2\bigl((1-T)P_5Q_5+T^2W\bigr).
\]
The barrier numerators are
\begin{equation}\label{prof:eq:Ssigma}
 S_\sigma=S_0+2\sigma\bigl(\alpha+(\beta-\alpha)T\bigr)Q_5^2,
 \qquad \sigma\in\{-1,+1\}.
\end{equation}
Every coefficient of \eqref{prof:eq:Ssigma} is rational.  If $\gamma_j$ are its Bernstein coefficients on $[a,b]$ and $\delta_j$ those of $Q_5^2$, then
\[
 \frac{S_\sigma(T)}{Q_5(T)^2}
 \ge\frac{\min_j\gamma_j}{\max_j\delta_j}.
\]
For the sufficient bound used here it is preferable to certify the
polynomial difference $S_\sigma-0.04Q_5^2>0$ directly on a finite rational
subdivision.  This avoids any interpretation of a lower Bernstein
certificate as an exact minimum of the rational function.  The exact rational coefficients for all 56 signed height and slope
checks are included in the accompanying data.

\section*{Acknowledgements}
During the preparation of this work the authors used ChatGPT (OpenAI) as an
auxiliary tool for exposition, code development and consistency checks in the
computer-assisted part of the proof.  All mathematical arguments,
computational certificates and the final manuscript were reviewed by the
authors, who take full responsibility for the contents. 

\section*{Acknowledgements}
During the preparation of this work the authors used ChatGPT (OpenAI) as an
auxiliary tool for exposition, code development and consistency checks in the
computer-assisted part of the proof.  All mathematical arguments,
computational certificates and the final manuscript were reviewed by the
authors, who take full responsibility for the contents.  

The complete computational archive, including sources, exact input data, interval outputs and reproduction scripts, will be made publicly available on the authors' webpage.

The research of M.~del Pino is supported by the Royal Society Research
Professorship grant RP-R1-180114 and by the ERC/UKRI Horizon Europe grant
ASYMEVOL, EP/Z000394/1. The research of J.~Wei is partially supported by
GRF of RGC of Hong Kong entitled ``On critical and supercritical Fujita
equation.

\end{document}